\documentclass[12pt]{article}

\usepackage{amsmath,amsthm,amsfonts,amssymb, color,xcolor,subcaption,graphicx,enumerate} 

\newtheorem{theorem}{Theorem}[section]
\newtheorem{corollary}[theorem]{Corollary}

\newtheorem{proposition}[theorem]{Proposition}
\newtheorem{lemma}[theorem]{Lemma}

\newtheorem{remark}[theorem]{Remark}

\def\cA{\mathcal{A}}
\def\cB{\mathcal{B}}

\def\cD{\mathcal{D}}

\def\cF{\mathcal{F}}
\def\cG{\mathcal{G}}
\def\cH{\mathcal{H}}

\def\cN{\mathcal{N}}

\def\cS{\mathcal{S}}
\def\cT{\mathcal{T}}

\def\bD{\mathbb{D}}
\def\bE{\mathbb{E}}
\def\bL{\mathbb{L}}

\def\bP{\mathbb{P}}
\def\bR{\mathbb{R}}

\def\s{\sigma}

\newcommand{\les}{\lesssim}

\begin{document}


\title{Gaussian fluctuations for the stochastic wave equation with drift}

\author{Raluca M. Balan\footnote{Corresponding author. University of Ottawa, Department of Mathematics and Statistics, 150 Louis Pasteur Private, Ottawa, Ontario, K1N 6N5, Canada. E-mail address: rbalan@uottawa.ca.} \footnote{Research supported by a grant from the Natural Sciences and Engineering Research Council of Canada.}
\and
William D. Stephenson \footnote{University of Ottawa, Department of Mathematics and Statistics, 150 Louis Pasteur Private, Ottawa, Ontario, K1N 6N5, Canada. E-mail address: wstep051@uottawa.ca.}
}

\date{October 18, 2026}
\maketitle

\begin{abstract}
\noindent In this article, we study the spatial fluctuations of the solution to the stochastic wave equation with a nonlinear drift and multiplicative Gaussian noise in dimensions one and two. The noise is white in time, and its spatial covariance is either integrable or given by a Riesz kernel; space-time white noise in dimension one is also included. We establish spatial ergodicity, convergence of the rescaled covariance, and prove a quantitative central limit theorem for the spatial fluctuations. We also obtain the corresponding functional central limit theorem. Our approach exploits the compact support of the wave kernel to obtain an estimate for the spatially integrated second Malliavin derivative, which is compatible with the total variation estimate in \cite{BNQZ}. Further arguments based on the Clark-Ocone formula are used to control the drift contributions.
\end{abstract}


\vspace{1mm}


\pagebreak

\tableofcontents

\section{Introduction}

Stochastic partial differential equations (SPDEs) are mathematical models for physical phenomena which evolve in space and time and are influenced by a random source. The solution $u(t,x)$ describes a physical measurement associated with such phenomena. For instance, for the stochastic wave equation in dimension 1, $u(t,x)$ is the displacement of a vibrating string at time $t$ and location $x$. 

A research area which has been growing rapidly in recent years is dedicated to the study of the asymptotic behaviour as $R \to \infty$ of the spatial average:
\begin{equation}
\label{def-FR}
F_R(t)=\int_{B_R} \Big(u(t,x) -\bE[u(t,x)]\Big) dx,
\end{equation}
where $B_R=\{|x|<R\}$. We denote $\sigma_R^2(t)={\rm Var}(F_R(t))$. It was discovered that $F_R(t)$ has the same asymptotic properties as the sum of zero-mean independent random variables with finite variance: it satisfies a strong law of large numbers, a central limit theorem, a law of the iterated logarithm, and even a large deviation principle. To understand why this is happening, consider a simple case, the stochastic heat equation with multiplicative noise (also known as the {\em parabolic Anderson model}), with flat initial condition:
\begin{equation}
\label{SHE}
\frac{\partial u}{\partial t} (t,x)=\frac{1}{2}\frac{\partial^2 u}{\partial x^2}(t,x)+u(t,x)\dot{W}(t,x), \quad \quad u(0,x)=1,
\end{equation}
driven by a space-time Gaussian white noise $W$, defined on a probability space $(\Omega,\cF,\bP)$. By analogy with Duhamel's principle, the solution satisfies, by definition:
\begin{equation}
\label{SHE-int}
u(t,x)=1+\int_0^t\int_{\bR}p_{t-s}(x-y) u(s,y) W(ds,dy),
\end{equation}
where $p_t(x)=(2\pi t)^{-1/2} e^{-\frac{|x|^2}{2t}}$ is the heat kernel. Iterating this equation, it follows that $u(t,x)=\lim_{n \to \infty} u_n(t,x)$ in $L^2(\Omega)$, where 
\[
u_n(t,x)= 1 + \sum_{k=1}^n \int_{0<t_1<\ldots <t_k<t} \int_{\bR^k} p_{t-t_k}(x-x_k) \ldots p_{t_2-t_1}(x_2-x_1) W(dt_1 dx_1) \ldots W(dt_k dx_k).
\]
So, $u_n(t,x)$ depends on the values of $W$ in $[0,t] \times \bR$. Next, consider the localized equation:
\[
u^{(\beta)}(t,x)=1+\int_0^t\int_{\{|x-y|<\sqrt{\beta t}\}}p_{t-s}(x-y) u^{(\beta)}(s,y) W(ds,dy),
\]
for some $\beta>0$. Then $u^{(\beta)}(t,x)=\lim_{n \to \infty} u_n^{(\beta)}(t,x)$ in $L^2(\Omega)$, where 
\[
u_n^{(\beta)}(t,x)= 1 + \sum_{k=1}^n \int_{0<t_1<\ldots <t_k<t} \int_{S_k} p_{t-t_k}(x-x_k) \ldots p_{t_2-t_1}(x_2-x_1) W(dt_1 dx_1) \ldots W(dt_k dx_k),
\]
and $S_k=\{|x-x_k|<\sqrt{\beta t},|x_k-x_{k-1}|<\sqrt{\beta t_{k}},\ldots, |x_2-x_1|<\sqrt{\beta t_2}\}$.
Then $u_n^{(\beta)}(t,x)$ depends on the values of $W$ in the region $[0,t] \times (x-n \sqrt{\beta t}, x + n\sqrt{\beta t})$. Since $W$ assumes independent values over disjoint space-time regions, $\{u_n^{(\beta)}(t,x_j)\}_{j\geq 1}$ are independent if $|x_i-x_j|> 2 n \sqrt{\beta t}$ for all $i\not=j$.

Suppose now that we choose $n=\beta=L_m$ for some $L_m>0$, and we denote
\[
S_m:=F_m(t)=\int_{-m}^m \big( u(t,x)-1 \big)dx \quad \mbox{and} \quad S_m^*=\int_{-m}^{m} \big( u_{L_m}^{(L_m)}(t,x)-1\big)dx.
\]
For suitable $L_m$, it is possible to show that $S_m$ and $S_m^*$ are asymptotically close, i.e. 
\[
\lim_{m \to \infty}\frac{\bE|S_m-S_m^*|^2}{m}=0.
\]
The scaling factor $m$ comes from the fact that $\sigma_m^2(t)={\rm Var}(S_m) \sim a_t m$ as $m \to \infty$, for some $a_t>0$. Hence, it is enough to study the asymptotic behaviour of $S_m^*$. For this, we use a blocking technique.
We write the interval $[-m,m]$ as a disjoint union of some ``big'' blocks $(B_i)_{i\leq q_m}$, separated by some ``small'' blocks $(B_i')_{i\leq q_m'}$.  Then $S_m^*=T_m+T_m'$, where
\[
T_m=\sum_{i=1}^{q_m} X_{i,m} \quad \mbox{with} \quad X_{i,m}=\int_{B_i} \big( u_{L_m}^{(L_m)}(t,x)-1\big)dx,
\]
and $T_m'$ has a similar expression using the blocks $B_i'$. If the separation between the blocks $B_i$'s is larger than $2n\sqrt{\beta t}=2\sqrt{t} L_m^{3/2}$, then $(X_{i,m})_{i\leq q_m}$ are independent, and the CLT for triangular arrays of independent variables can be invoked to show that 
\[
\frac{T_m}{\sqrt{m}} \stackrel{d}{\to} N(0,a_t).
\]
Finally, for suitable $L_m$, the small blocks are negligible (e.g. $T_m'/\sqrt{m} \to 0$ in $L^2(\Omega)$). So,
\[
\frac{S_m^*}{\sqrt{m}} \stackrel{d}{\to} N(0,a_t) \quad \mbox{and} \quad \frac{S_m}{\sqrt{m}} \stackrel{d}{\to} N(0,a_t).
\]

This heuristic argument for proving the CLT:
\begin{equation}
\label{CLT}
\frac{F_R(t)}{\sigma_R(t)} \stackrel{d}{\longrightarrow} N(0,1) \quad \mbox{as $R \to \infty$}
\end{equation}
 was hinted at in \cite{HNV20}, being inspired by \cite{CJK13}. The nice thing about this argument is that it is very robust: it also works for equation \eqref{SHE} with $u(t,x)$ replaced by $\sigma(u(t,x))$, where $\sigma$ is a Lipschitz function; it can serve as justification for other limit theorems (see \cite{li-zhang23} for LIL and \cite{ebina25} for large deviations); it can be applied for correlated noise in space, as long as the correlation decays fast enough so that the variables $(X_{i,m})_{i\leq q_m}$ remain weakly dependent.
 Moreover, a similar method can be applied for other equations.
 For instance, in the case of the stochastic wave equation, the solution satisfies \eqref{SHE-int}, with the heat kernel $p_t(x)$ replaced by the wave kernel $G_t(x)=\frac{1}{2}1_{\{|x|<t\}}$.
(In this case, there is no need for localization, since the wave kernel has compact support.)

\medskip

Although heuristic arguments are always very valuable to guide our intuition, 
the major breakthrough of \cite{HNV20} is not this argument which explains {\em why} we expect that \eqref{CLT} holds, but the method which shows {\em how} to obtain a convergence rate in this CLT. To explain this, we need to recall {\em Stein's equation} (for a fixed function $h:\bR \to \bR$):
\[
f'(x)-xf(x)=h(x)-\bE[h(Z)], \quad x \in \bR,
\]
where $Z \sim N(0,1)$. Under some conditions, this equation has a unique solution $f=f_h$. Replacing $x$ by a random variable $F$ and taking the expectation, we get: 
\[
\bE[f'(F)]-\bE[Ff(F)]=\bE[h(F)]-\bE[h(Z)].
\]
The right side is used to measure the nearness between $F$ and $Z$, using various distances. For instance,  for the total variation distance, if we let $\cF_{TV}=\{1_B; B \in \cB(\bR)\}$, then
\[
d_{TV}(F,Z)=\sup_{h \in \cF_{TV}}\big| \bE[h(F)]-\bE[h(Z)] \big|=\sup_{h \in \cF_{TV}}\big|\bE[f_h'(F)]-\bE[Ff_h(F)]\big|.
\]

As explained at length in the monograph \cite{NP}, a natural connection with Malliavin calculus arises. If $F$ is Malliavin differentiable with respect to an isonormal Gaussian process $W=\{W(\varphi)\}_{\varphi \in \cH}$ (for instance the driving noise of an SPDE), and $F=\delta(v)$ for {\em some} process $v$, where $\delta$ is the Skorohod integral with respect to $W$ (which is also the adjoint of the Malliavin derivative $D$ with respect to $W$), then
\[
\bE[Ff(F)]=\bE \big(\langle v, D f(F) \rangle_{\cH}\big) =\bE[f'(F) \langle DF,v \rangle_{\cH}],
\]
by the duality between $D$ and $\delta$, and the chain rule for $D$. Since $|f_h'(x)|\leq 2$ for all $x\in \bR$,
\[
d_{TV}(F,Z) \leq 2 \sup_{h \in \cF_{TV}} \bE \big|1-\langle DF,v \rangle_{\cH}\big|.
\]
By duality, $\bE(\langle DF,v \rangle_{\cH})=\bE(F^2)$. So, if $\bE(F^2)=1$, we obtain the {\em Stein-Malliavin bound}:
\begin{equation}
\label{SM-eq}
d_{TV}(F,Z) \leq 2\sqrt{{\rm Var} (\langle DF,v \rangle_{\cH})}.
\end{equation}

Subsequently, motivated by the fact that {\em any} centered random variable $F$ can be written as $F=\delta(v)$ with $v=-D L^{-1} F$, where  $L^{-1}$ is the pseudo-inverse of the Ornstein-Uhlenbeck operator, Proposition 3.2 of \cite{vidotto20} provided a useful bound for
\[
{\rm Var} (\langle DF,v \rangle_{\cH})=\bE\Big[\Big({\rm Var}(F)-\langle DF, -D L^{-1} F \rangle_{\cH}\Big)^2\Big]
\]
when $F$ is twice Malliavin differentiable ($F \in \bD^{2,4}$) and $\cH=L^2(A,\cA,\mu)$. This bound combined with \eqref{SM-eq} is called the {\em second-order (Gaussian) Poincar\'e inequality}. 

\medskip

In the context of SPDEs, relation \eqref{SM-eq} has been applied in many articles to prove that
\[
d_{TV}\left(\frac{F_R(t)}{\sigma_R(t)},Z\right) \leq R^{-\varepsilon}
\]
for some $\varepsilon>0$, yielding a {\em quantitative} CLT; see 
\cite{HNVZ20,NXZ22,NSZ21,BY23} for the stochastic heat equation, and \cite{BNZ,DNZ20,NZ22,BNQZ,BY24} for the stochastic wave equation. Moreover, these references proved also the corresponding functional limit results. The methods used in these references can be summarized as follows. When the noise is white, we write $F_R(t)=\delta(v_{R,t})$, where $v_{R,t}$ is the process which arises naturally from the integral form of the equation (for instance \eqref{SHE-int}) by taking the $dx$ integral on $B_R$. In this case, the Stein-Malliavin bound \eqref{SM-eq} is sufficient, combined with a pointwise moment estimate for the first Malliavin derivative in terms of the fundamental solution of the equation. 
But when the noise is colored in time, one needs to appeal to the second-order Poincar\'e inequality, together with an estimate for the second order Malliavin derivative.

\medskip

In all the references mentioned above, the equation does not contain a drift term. The addition of the drift does not affect the heuristic argument presented above, provided that the Lebesgue space-time integral is also localized. But the presence of the drift produces a mixture of Lebesgue and stochastic integral terms which pose difficulties that are far from negligible.
In the recent article \cite{BS26}, the first author and M. Salins proved the quantitative CLT for the spatial average of the stochastic heat equation with drift, driven by space-time Gaussian noise, using the second-order Poincar\'e inequality together with a new estimate for products of heat kernels, which is used to control the second Malliavin derivative. 

\medskip

In the present article, we continue the line of investigation initiated in \cite{BS26} and we study the QCLT problem for the stochastic wave equation in dimension $d \in \{1,2\}$. More precisely, we consider the equation:
\begin{align}
\label{SWE}
	\begin{cases}
		\dfrac{\partial^2 u}{\partial t^2} (t,x)
		=  \Delta u (t,x)+\sigma\big(u(t,x)\big) \dot{W}(t,x)+b\big(u(t,x) \big), \ 
		t>0, \ x \in \bR^d \ (d\leq 2), \\
		u(0,x) = 1, \dfrac{\partial u}{\partial t} (0,x)=0, \quad x \in \bR^d,
	\end{cases}
\end{align}
driven by a spatially-homogeneous Gaussian noise which is white in time. The goal of this article is to show that $F_R(t)/\sigma_R(t)$ converges in distribution to $N(0,1)$ when $R \to \infty$, and give a precise rate of this convergence in the total variation distance.

The noise is given by a zero-mean Gaussian process  $W=\{W(\varphi);\varphi \in \cD(\bR_{+} \times \bR^d)\}$ defined on a complete probability space $(\Omega,\cF,\bP)$, with covariance
\[
\bE[W(\varphi)W(\psi)]=\int_0^{\infty} \int_{\bR^d} \int_{\bR^d}\varphi(t,x)\psi(t,y)f(x-y)dxdydt=:\langle \varphi,\psi \rangle_{\cH},
\] 
where $f:\bR^d \to [0,\infty]$ is a tempered non-negative definite function.  By Bochner-Schwartz theorem, $f$ is the  Fourier transform of a tempered measure $\mu$ on $\bR^d$. Hence, for any $\varphi,\psi \in \cS(\bR^d)$,
\[
\int_{\bR^d} \int_{\bR^d}\varphi(x)\psi(y) f(x-y)dxdy=\int_{\bR^d}\cF \varphi(\xi) \overline{\cF \psi(\xi)}\mu(d\xi),
\]
where $\cS(\bR^d)$ is the space of rapidly decreasing functions on $\bR^d$, $\cF \varphi(\xi)=\int_{\bR^d} e^{-i\xi \cdot x} \varphi(x)dx$ is the Fourier transform of $\varphi$, and $x\cdot y$ denotes the usual inner product in $\bR^d$. Let $\cH$ denote the completion of $\cD(\bR_{+} \times \bR^d)$ with respect to the inner product $\langle \cdot,\cdot \rangle_{\cH}$. 

Throughout, we assume that $\sigma,b \in C^2(\bR)$ with $b',b'',\sigma',\sigma''$ bounded, and the Lebesgue measure is absolutely continuous with respect to $\mu$.

We assume that $W$ is defined on a complete probability space $(\Omega,\cF,\bP)$. We denote $W(A)=W(1_A)$ for any bounded Borel set $A$ in $\bR_{+}\times \bR^d$.
We denote by $(\cF_t)_{t\geq 0}$ the (augmented) filtration induced by $W$:
\begin{equation}
\label{def-Ft}
\cF_t=\sigma\{ W([0,s] \times A); s \in [0,t], A \in \cB_b(\bR^d) \} \vee \cN,
\end{equation}
where $\cB_b(\bR^d)$ is the class of bounded Borel sets in $\bR$, and $\cN$ are $\bP$-negligible events.

A random field $\{\Phi(t,x);t\geq 0,x\in \bR^d\}$ is {\em predictable} if it is measurable with respect to the predictable $\sigma$-field on $\Omega \times \bR_{+} \times \bR^d$, which is the minimal $\sigma$-field with respect to which all elementary processes are measurable. An {\em elementary process} is a linear combination of processes of the form
\[
\Phi(t,x)=Y \mathbf 1_{(a,b]}(t) \mathbf 1_{A}(x)
\]
where  $0\leq a<b$, $A \in \cB_b(\bR^d)$ and $Y$ is $\cF_a$-measurable. 

A predictable process $\{u(t,x);t\geq 0,x\in \bR^d\}$ is called a {\em (mild) solution} of \eqref{SWE}
if it satisfies the integral equation:
\begin{align}
\nonumber
u(t,x) & =1+\int_0^t \int_{\bR^d}G_{t-s}(x-y)\sigma\big(u(s,y)\big) W(ds,dy) \\
\label{mild soln}
& \quad +\int_0^t \int_{\bR^d} G_{t-s}(x-y) b\big(u(s,y)\big)dsdy,
\end{align}
where the stochastic integral is interpreted in the It\^o sense, and $G$ is the fundamental solution of the wave equation:
\begin{align*}
G_t(x)&=\frac{1}{2}\mathbf 1_{\{|x|<t\}} \quad \quad \quad \quad \quad \quad \mbox{if} \quad d=1\\
G_t(x)& =\frac{1}{2\pi} \frac{1}{\sqrt{t^2-|x|^2}} \mathbf 1_{\{|x|<t\}} \quad \mbox{if} \quad d=2.
\end{align*}
Here, $|\cdot|$ denotes the Euclidean norm. We will use the fact that for any $t>0$,
\begin{equation}
\label{int-G}
\int_{\bR^d}G_t(x)dx=t \quad \mbox{for} \quad d \in \{1,2\}.
\end{equation}

In \cite{dalang99}, it was proved that if $b$ and $\sigma$ are globally Lipschitz, and $\mu$ satisfies {\em Dalang's condition}:
\[
(D) \quad
\int_{\bR^d}\frac{1}{1+|\xi|^2}\mu(d\xi)<\infty,
\]
then equation \eqref{SWE} has a unique solution which satisfies: for any $p\geq 2$,
\begin{equation}
\label{def-KTp}
K_{T,p}:=\sup_{(t,x) \in [0,T] \times \bR^d}\bE|u(t,x)|^p<\infty \quad \mbox{for any} \quad T>0.
\end{equation}
In dimension $d=1$, Dalang's condition always holds. 

Since the noise is spatially homogeneous and the initial condition is constant, the process $\{u(t,x)\}_{x\in \bR^d}$ is strictly stationary.

\medskip

To formulate our main results, we now specify the assumptions on the spatial covariance of the noise. We consider two cases: \\
(i) $f\in L^1(\bR^d)$ and (D) holds, or $d=1$ and the noise is white; if $d = 2$, we assume in addition that $f \in L^{\ell}(\bR^2)$ for some $\ell > 1$;
\\
(ii) $f(x)=|x|^{-\beta}$ is the Riesz kernel of order $\beta \in (0,d\wedge 2)$. In this case,
\[
\int_{\bR^d}\varphi(x)|x|^{-\beta}dx=c_{d,\beta} \int_{\bR^d}\cF \varphi(\xi) |\xi|^{-(d-\beta)}d\xi, \quad \mbox{for all} \quad \varphi \in \cS(\bR^d),
\]
where
\[
c_{d,\beta}=\pi^{-\frac{d}{2}} 2^{-\beta} \frac{\Gamma(\frac{d-\beta}{2})}{\Gamma(\frac{\beta}{2})}.
\]

\medskip

Below are the main results of this article.

\begin{theorem}
\label{ergodic-th}
For any $t>0$, the process $\{u(t,x)\}_{x\in \bR^d}$ is ergodic.
\end{theorem}

\begin{theorem}[Limiting covariance]
\label{cov-th}
For any $t_1,t_2>0$,
\begin{equation}
\label{def-K}
K(t_1,t_2):=\lim_{R \to \infty}\frac{1}{R^{\gamma}}\bE[F_R(t_1) F_R(t_2)] \quad \mbox{is finite},
\end{equation}
where
\begin{equation}
\label{def-gamma}
\gamma:=\left\{
\begin{array}{ll} d & \mbox{in case (i)} \\
2d-\beta & \mbox{in case (ii) }.
\end{array} \right.
\end{equation}
In particular, $\sigma_R^2(t) \sim K(t,t) R^{\gamma}$ as $R \to \infty$, for any $t>0$.
\end{theorem}

\begin{theorem}[Quantitative Central Limit Theorem]
\label{QCLT}
For any $t>0$, there exists a constant $C_t>0$ depending on $t$ such that for any $R>0$,
\[
d_{TV}\left( \frac{F_R(t)}{\sigma_R(t)},Z\right) \leq C_t R^{\frac{\gamma-2d}{2}},
\]
where $Z$ is a standard normal random variable, $\gamma$ is given by \eqref{def-gamma}, and $d_{TV}$ is the total variation distance.
\end{theorem}

\begin{theorem}[Functional Central Limit Theorem]
\label{FCLT}
The process $\{F_R(t)\}_{t\geq 0}$ has a $\delta$-H\"older continuous modification (denoted also $F_R$), for any $\delta \in (0,1/d)$. The process $\{R^{-\gamma/2} F_R(t)\}_{t\geq 0}$ converges in distribution in $C[0,\infty)$ as $R \to \infty$ to a zero-mean Gaussian process $\{\cG(t)\}_{t\geq 0}$ with covariance 
\[
\bE[\cG(t_1) \cG(t_2)]=K(t_1,t_2),
\]
where $\gamma$ is given by \eqref{def-gamma} and $K(t_1,t_2)$ is given by \eqref{def-K}.
\end{theorem}

The principal difficulty in extending the drift-free results to the present setting comes from the drift term. Although the pointwise estimate for the first Malliavin derivative remains valid, as shown in Theorem \ref{key-Du}, the usual bound for the second derivative by a product of fundamental solutions cannot be used for a general nonlinear drift. Indeed, differentiating the drift twice produces a convolution involving $b''(u)D_{r,z}uD_{\theta,w}u$, which does not have the time-ordered product structure available in the linear Anderson setting. Establishing suitable estimates for the second Malliavin derivative is very delicate in dimension $d=2$, because of the singularity of the wave kernel at the boundary of its support. Instead of seeking an analogous pointwise bound, we exploit the compact support of $G$ and the resulting localization of the Malliavin derivatives to provide an estimate for the integral of the second Malliavin derivative (see Theorem \ref{key-D2u-int}). This estimate turns out to be sufficient for the application of the second order Poincar\'e inequality, and produces the desired rate in Theorem \ref{QCLT}.

The drift also introduces difficulties in the limiting covariance and functional limit arguments. The covariance analysis requires separate treatment of the drift-drift term and mixed contributions, while tightness requires temporal-increment estimates at the fluctuation scale $R^{\gamma/2}$. For the latter, we apply the Clark-Ocone formula to the centered drift increments, converting them into stochastic integrals whose moments can be controlled by the Burkholder-Davis-Gundy inequality and Theorem \ref{QCLT}. This yields bounds compatible with the normalization of the spatial average, and completes the tightness argument.

\medskip

This article is organized as follows. Section \ref{section-prelim} contains preliminary material. Section \ref{section-Mal} establishes key estimates for the Malliavin derivatives of the solution. Section \ref{section-main} contains the proofs of Theorems \ref{ergodic-th}, \ref{cov-th}, \ref{QCLT} and \ref{FCLT}. The appendices contain some auxiliary results.

\medskip

{\em Notation:} We use the notation $a_R \les b_R$ if $a_R \leq C b_R$ where $C$ is a constant that does not depend on $(a_R,b_R)$. We denote by $\| \cdot \|_p$ the norm in $L^p(\Omega)$ for any $p\geq 1$.

\section{Preliminaries}
\label{section-prelim}

We first introduce the notation and auxiliary results that will be used throughout the paper.

\subsection{Inequalities}

The following is a collection of several Hardy--Littlewood--Sobolev--type inequalities that will be used throughout.

\medskip

If $f \in L^{\ell}(\bR^d)$ for some $\ell \geq 1$, using H\"older's inequality and Young's inequality, we have:
\begin{align}
\nonumber
\left| \int_{\bR^d}\int_{\bR^d}\varphi(x)\psi(y) f(x-y)dxdy \right| & \leq \|\varphi\|_{L^{\frac{2\ell}{2\ell-1}}(\bR^d)} \| \psi * f\|_{L^{2\ell}(\bR^d)} \\
\label{LHS-Lell}
& \leq \|\varphi\|_{L^{\frac{2\ell}{2\ell-1}}(\bR^d)} \|\psi\|_{L^{\frac{2\ell}{2\ell-1}}(\bR^d)} \|f\|_{L^{\ell}(\bR^d)}.
\end{align}

If $f(x)=|x|^{-\beta}$ with $\beta \in (0,d)$, 
by Hardy-Littlewood-Sobolev inequality, we have:
\begin{equation*}
\left|\int_{\bR^d} \int_{\bR^d}\varphi(x) \psi(y) |x-y|^{-\beta} dxdy\right| \leq c_{\beta,d} \|\varphi\|_{L^{\frac{2d}{2d-\beta}}(\bR^d)} \|\psi\|_{L^{\frac{2d}{2d-\beta}}(\bR^d)},
\end{equation*}
where $c_{\beta,d}>0$ is a constant which depends on $(\beta,d)$; see e.g. Remark B.3 of \cite{BS26} for $d=1$, and Section 3.3 of \cite{BNZ} for $d=2$. 

To shorten the exposition, we combine the above inequalities in two distinct ways. First we have:
\begin{equation}
\label{LHS-q}
\left|\int_{\bR^d} \int_{\bR^d}\varphi(x) \psi(y) f(x-y) dxdy\right| \leq c_{f,d} \|\varphi\|_{L^{2q}(\bR^d)}\|\psi\|_{L^{2q}(\bR^d)},
\end{equation}
where $c_{f,d}>0$ is a constant that depends on $(f,d)$, and
\begin{equation}
\label{def-q}
q:=\left\{
\begin{array}{ll} 
1 & \mbox{if $d=1$ and $f \in L^1(\bR)$ or the noise is white}, \\
\frac{\ell}{2\ell-1} & \mbox{if $d = 2$ and $f \in L^1(\bR^2) \cap L^{\ell}(\bR^2)$ for some $\ell >1$}, \\
\frac{d}{2d-\beta} & \mbox{if $f(x)=|x|^{-\beta}$ for some $\beta \in (0,d)$}.
\end{array} \right.
\end{equation}
In the last two cases, $q \in (\frac12,1)$. Second, we have:
\begin{equation}
\label{LHS-q'}
\left|\int_{\bR^d} \int_{\bR^d}\varphi(x) \psi(y) f(x-y) dxdy\right| \leq c_{f,d} \|\varphi\|_{L^{2q'}(\bR^d)}\|\psi\|_{L^{2q'}(\bR^d)},
\end{equation}
where
\begin{equation}
\label{def-q'}
q':=\frac{d}{\gamma}=\left\{
\begin{array}{ll} 
1 & \mbox{if $f \in L^1(\bR^d)$ or $d=1$ and the noise is white}, \\
\frac{d}{2d-\beta} & \mbox{if $f(x)=|x|^{-\beta}$ for some $\beta \in (0,d)$}.
\end{array} \right.
\end{equation}
Inequalities \eqref{LHS-q} and \eqref{LHS-q'} are obviously the same, when $f(x)=|x|^{-\beta}$.
The subtle difference between them is that \eqref{LHS-q'} only requires $f \in L^1(\bR^d)$ even for $d=2$, i.e. we will use \eqref{LHS-q'} for arguments for which the hypothesis $f \in L^{\ell}(\bR^2)$ for some $\ell>1$ is {\em not} needed. Moreover, the exponent $q'$ on the right side of \eqref{def-q'} depends on the parameter $\gamma$, which appears in the QCLT rate, and in the normalization in the FCLT. 

\subsection{Properties of $G$}

In this section, we include some properties of the wave kernel $G$. First, note that
\begin{equation}
\label{int G p}
 \int_{\bR^d} G^p_t(x) dx = 
 \begin{cases} 
  2^{1-p}t & \mbox{for any $p >0$, if $d=1$} \\
  \frac{(2\pi)^{1-p}}{2-p}t^{2-p} & \mbox{for any $p\in (0, 2)$, if $d=2$}. \\
 \end{cases}
\end{equation}
The next result gives an explicit expression for the convolution of two wave kernels.
\begin{lemma}
\label{lemG*G}
Let $d\in \{1,2\}$. For any $t,s>0$ and $x\in \bR^d$, we have:
\begin{equation}
\label{G*G}
(G_t*G_s)(x)=\frac{1}{2}\int_{|t-s|}^{t+s} G_{\theta}(x) d\theta.
\end{equation}
\end{lemma}

\begin{proof}
We show that both sides have the same Fourier transforms. On the left side,
\begin{align*}
\cF(G_t*G_s)(\xi)&= \cF G_t(\xi)\cF G_s(\xi)=\frac{\sin(t|\xi|) \sin(s|\xi|)}{|\xi|^2}.
\end{align*}
On the right side,
\begin{align*}
\frac{1}{2}\int_{|t-s|}^{t+s} \cF G_{\theta}(\xi) d\theta=\frac{1}{2}\int_{|t-s|}^{t+s} \frac{\sin(\theta|\xi|)}{|\xi|} d\theta=\frac{\cos\big(|t-s||\xi|\big)- \cos \big( (t+s)|\xi|\big)}{2|\xi|^2}.
\end{align*}
\end{proof}
For any $t>0$ and $x \in \bR^d$, we denote $K_t(x) := (G_t*G_t)(x)$. By Lemma \ref{lemG*G},
\[
K_t(x) =\frac{1}{2}\int_0^{2t} G_{\theta}(x)d\theta. 
\]
Hence, $t \mapsto K_t(x)$ is non-decreasing on $\bR_{+}$, and $K_t(x)=0$ if $|x| \geq 2t$. By direct calculation,
\begin{align}
\label{def-K1}
 K_t(x) & = \frac{1}{4} (2t - |x|)  1_{\{|x| < 2t\}}  \qquad \text{if} \ d = 1 \\
 \label{def-K2}
 K_t(x) & = \frac{1}{4\pi}  \log\left(\frac{2t+\sqrt{4t^2-|x|^2}}{|x|}\right)  1_{\{|x|<2t\}}\qquad \text{if} \ d = 2.
\end{align}
The following lemma provides an upper bound similar to inequality \eqref{LHS-q}, but for the kernel $fK_t^{\alpha}$ replacing $f$.
\begin{lemma} \label{fK int}
Let $ t \in [0, T]$ and $\alpha>0$. Then 
\begin{enumerate}
 \item[(i)] If $d = 1$ and $f \in L^1(\bR)$ or the noise is white, then
\[
\int_{\bR} \int_{\bR} \varphi(x)\psi (y) f(x-y) K_t^{\alpha}(x-y) dx dy  \leq T^{\alpha} \|f\|_{L^1(\bR)} \|\varphi\|_{L^{2}(\bR)}\|\psi\|_{L^{2}(\bR)}.
\]
\item[(ii)] If $d\in \{1,2\}$ and either (a) $f \in L^\ell(\bR^d)$ for some $\ell >1$, or (b) $f(x) = |x|^{-\beta}$ for some $\beta \in (0,d)$, then 
\[
\int_{\bR^d} \int_{\bR^d} \varphi(x)\psi(y) f(x-y)K_t^{\alpha}(x-y)  dx dy \leq C_{T,\alpha,p',q} \|\varphi\|_{L^{2p'}(\bR^d)}\|\psi\|_{L^{2p'}(\bR^d)},
\]
for any $\alpha > 0$ and $p'\in (q,1)$, where $q=\frac{\ell}{2\ell-1}$ in case (a), and $q=\frac{d}{2d-\beta}$ in case (b). 
\end{enumerate}
\end{lemma}

\begin{proof} 
(i) By \eqref{def-K1}, it follows that $K_t(x) \leq T$ for any $x \in \bR$ and $t \in [0,T]$. Therefore, 
\[
 \int_{\bR} \int_{\bR} \varphi(x) \psi(y)  f(x-y)K_t^{\alpha}(x-y) dx dy \leq  T^{\alpha} \int_{\bR} \int_{\bR} |\varphi(x)| |\psi(y)| f(x-y) dx dy.
\]
The conclusion follows by applying \eqref{LHS-Lell} with $\ell=1$.  

(ii) Using \eqref{def-K1} and \eqref{def-K2}, it follows that for all $\varepsilon > 0$,
\begin{equation}
\label{bound-K}
 K_t(x) \leq C_{T,\epsilon} |x|^{-\varepsilon}.
\end{equation}

We fix $p' \in (q,1)$. We apply H\"older's inequality, followed by Young's inequality:
\begin{align*}
I &:=\int_{\bR^d} \int_{\bR^d}  \varphi(x) \psi(y) f(x-y)K_t^{\alpha}(x-y) dx dy 
= \int_{\bR^d} (\varphi * fK_t^{\alpha})(y) \psi(y) dy \\
& \quad \leq \| \varphi* fK_t^\alpha  \|_{L^a(\bR^d)} \|\psi \|_{L^{2p'}(\bR^d)}  \quad \text{with}\  \frac{1}{a} + \frac{1}{2p'} = 1, \\
& \quad \leq  \| fK_t^\alpha\|_{L^{b}(\bR^d)}  \|\varphi \|_{L^{2p'}(\bR^d)}\|\psi \|_{L^{2p'}(\bR^d)} \quad  \text{with } \  \frac{1}{b} +\frac{1}{2p'} = 1 + \frac{1}{a} = 2 - \frac{1}{2p'} .
\end{align*}
Note that $b=\frac{p'}{2p'-1}>1$ since $p'\in (\frac12,1)$. 

To estimate $\|K_t^{\alpha}\|_{L^{c}(\bR^d)}$, we consider separately the two cases.
In case (a), we apply H\"older's inequality with $\frac{1}{p_1}+\frac{1}{p_2}=1$ and $p_1=\frac{\ell}{b}>1$. (This means that $\frac{1}{\ell}<\frac{1}{b}=2 - \frac{1}{p'}$, i.e. $p'>\frac{\ell}{2\ell -1}=q$.) Then

\begin{align*}
\| fK_t^{\alpha}\|_{L^{b}(\bR^d)}  & \leq \| f\|_{L^\ell(\bR^d)} \|K_t^{\alpha}\|_{L^{c}(\bR^d)} \qquad \text{with } \frac{1}{\ell} + \frac{1}{c} = \frac{1}{b}.
\end{align*}

We use \eqref{bound-K} and the compact support of $K_t$. Then,
$\|K_t^{\alpha}\|_{L^{c}(\bR^d)} \leq C_{T,\varepsilon}\left(\int_{|x| < 2t} |x|^{-\alpha \varepsilon c} dx \right)^{1/c}< \infty$, provided that we choose $\varepsilon > 0$ so that $\alpha \varepsilon  c< d$. This implies that
\[
I \leq C_{T,\varepsilon,\alpha,c,d} \| f\|_{L^\ell(\bR^d)} \|\varphi\|_{L^{2p'}(\bR^d)}\|\psi\|_{L^{2p'}(\bR^d)}. 
\]

In case (b), using \eqref{bound-K} and the compact support of $K_t$, we have:
\[
\|fK_t^{\alpha}\|_{L^b}^b \leq C_{T,\varepsilon} \int_{\{|x|<2t\}}  |x|^{-\beta b-\alpha\varepsilon  b} dx<  \infty,
\]
provided that we choose $\varepsilon>0$ such that $(\beta+\alpha\varepsilon ) b< d$. This means that 
$$2-\frac{1}{p'}=\frac{1}{b}>\frac{\beta+\alpha \varepsilon}{d} \quad \mbox{i.e.} \quad \frac{\alpha \varepsilon}{d}< \frac{2d-\beta}{d}-\frac{1}{p'}.$$
This is possible since $p'>\frac{d}{2d-\beta} =q$. Then,
\[
I \leq C_{T,\varepsilon,\alpha,\beta,b,d} \|\varphi\|_{L^{2p'}(\bR^d)}\|\psi\|_{L^{2p'}(\bR^d)}. 
\]
\end{proof}

\bigskip 

We also have the following inequalities for the wave kernel $G$ integrated over a ball of radius $R$. 
For any $R>0$, $0\leq r \leq t$ and $y \in \bR^d$, we denote:
\begin{equation}
\label{def-var}
\varphi_{t,R}(r,y)=\int_{B_R} G_{t-r}(x-y)dx.
\end{equation}

Observe that, $0 \leq \varphi_{t,R}(r,y) \leq t-r$ and $\varphi_{t,R}(r,y)$ is supported on $\{|y| \leq R + t - r \}$. In particular, when $R \geq t$, $\varphi_{t,R}(r,y)$ is supported on $B_{2R}$. Therefore, for any $p\geq 1$, $0\leq r\leq t$ and $R>t$,
\begin{equation}
\label{est}
 \| \varphi_{t,R}(r,\cdot) \|_{L^p(\bR^d)} \les R^{d/p} (t-r).
\end{equation}
From the proof of Lemma 4.3 in \cite{NZ22}, we know that
\[
\Big|\varphi_{t,R}(r,y) - \varphi_{s,R}(r,y) \Big| \les (t-s)^{1/d} \mathbf 1_{\{|y|\leq R+t\}} \quad \mbox{for any $0\leq r\leq s\leq t$}.
\]
As a consequence, for any $p\geq 1$, $0\leq r\leq s\leq t$ and $R>t$,
\begin{equation}
\label{diff est}
 \| \varphi_{t,R}(r,\cdot) - \varphi_{s,R}(r,\cdot)\|_{L^p(\bR^d)} \les R^{d/p} (t-s)^{1/d}.
\end{equation}

\subsection{Malliavin Calculus}

In this section, we include some background material about Malliavin calculus.

\medskip

$\bullet$ {\bf It\^o integral.}
The It\^o integral $\int_0^t \int_{\bR^d} g(s,x)W(ds,dx),t\geq 0$ is a martingale defined for any  $(\cF_t)_{t\geq 0}$-predictable process $g$ such that $\bE\int_0^T \int_{\bR^d} \int_{\bR^d} |g(t,x)g(t,y)|f(x-y)dxdydt<\infty$ for any $T>0$. By Burkholder-Davis-Gundy and Minkowski's inequalities, for any $p\geq 2$,
\begin{equation}
\label{BDG}
\left\| \int_0^T \int_{\bR^d} g(t,x)W(dt,dx) \right\|_p^2 \leq z_p^2 \int_0^T \int_{\bR^d}\int_{\bR^d}\|g(t,x)\|_p \|g(t,y)\|_p f(x-y)dxdydt, 
\end{equation}
where $z_p>0$ is a constant depending on $p$.

\medskip

$\bullet$ {\bf Malliavin derivative.}
Let $\cS$ be the class of ``smooth'' random variables:
\begin{equation}
\label{form-F}F=f(W(\varphi_1),\ldots, W(\varphi_n)),
\end{equation} where $f \in C_{b}^{\infty}(\bR^n)$, $\varphi_i \in \cH$,
$n \geq 1$, and
$C_b^{\infty}(\bR^n)$ is the class of bounded $C^{\infty}$-functions
on $\bR^n$, whose partial derivatives of all orders are bounded. The
{\em Malliavin derivative} of $F$ of the form (\ref{form-F}) is the
$\cH$-valued random variable given by:
$$D_{t,x}F:=\sum_{i=1}^{n}\frac{\partial f}{\partial x_i}(W(\varphi_1),\ldots,
W(\varphi_n))\varphi_i(t,x), \quad \mbox{for $t>0$ and $x\in \bR^d$}.$$ We endow $\cS$ with the semi-norm
$\|F\|_{\bD^{1,2}}=(\bE|F|^2+\bE\|D F \|_{\cH}^{2})^{1/2}$. The
operator $D$ can be extended to the space $\bD^{1,2}$, the
completion of $\cS$ with respect to $\|\cdot \|_{\bD^{1,2}}$.

\medskip

$\bullet$ {\bf Skorohod integral.}
The {\em divergence operator} $\delta$ is the adjoint of
the operator $D$. If $u \in {\rm Dom} (\delta)$, then $\delta(u)\in L^2(\Omega)$ is
characterized by the duality relation:
\begin{equation}
\label{duality}
\bE(F \delta(u))=\bE\left(\langle DF,u \rangle_{\cH}\right), \quad
\forall F \in \bD^{1,2}.
 \end{equation}
In particular, $\bE[\delta(u)]=0$.  For any $u \in \mbox{Dom} (\delta)$,
we say that $\delta(u)$ is the {\em Skorohod integral} of $u$. If $g$ is $(\cF_t)_{t\geq 0}$-predictable, then $\delta(g)=\int_0^{\infty} \int_{\bR^d} g(t,x)W(dt,dx)$.

\medskip

We will need the following properties, which are proved similarly to the white noise case discussed in \cite{nualart06}.

\begin{proposition}[Chain Rule] 
For any $F \in \bD^{1,2}$ and $\varphi \in C^1(\bR)$ with $\varphi'$ bounded,
\begin{equation}
\label{chain}
D\varphi(F)=\varphi'(F)DF.
\end{equation}
\end{proposition}

\begin{proposition}
If $F\in \bD^{1,2}$ is $\cF_t$-measurable, then 
\begin{equation}
\label{Mal-zero}
D_{r,z}F=0 \quad \mbox{for all $r>t$ and $z \in \bR^d$}.
\end{equation}
\end{proposition}

\begin{proposition}[Clark-Ocone formula]
If $F \in \bD^{1,2}$ and $F$ is $\cF_T$-measurable, then
\begin{equation}
\label{clark-ocone}
F=\bE(F)+\int_0^T \int_{\bR^d} \bE[D_{t,x}F|\cF_t] W(dt,dx).
\end{equation}
\end{proposition}


\begin{proposition}[Poincar\'e inequality]
For any $F,G \in \bD^{1,2}$,
\begin{equation}
\label{Poincare}
|{\rm Cov}(F,G)| \leq \int_{\bR_+} \int_{\bR^{2d}} \|D_{r,y}F \|_2 \|D_{r,y'}G\|_2 f(y-y') dydy'dr.
\end{equation}
\end{proposition}

Let $\bL^{1,2}$ be the class of processes $u \in L^2(\Omega;\cH)$ such that $u(t,x)\in \bD^{1,2}$ for all $(t,x)\in \bR_{+}\times \bR^d$, and there exists a measurable modification of
$\{D_{t,x}u(s,y);(t,x),(s,y)\}$ such that
\[
\bE \int_{(\bR_+ \times \bR^d \times \bR^d)^2} D_{r,y}u(t,x)D_{r,y'}u(t,x') f(x-x')f(y-y') dxdx'dydy'drdt < \infty.
\]

\begin{proposition}[Heisenberg commutation rule] 
If $u \in \bL^{1,2}$, $\{D_{t,x}u(s,y)\}_{(s,y)} \in {\rm Dom}(\delta)$ for almost all $(t,x)$, and there exists a measurable modification of $\{\int D_{t,x}u(s,y)W(\delta s,\delta y)\}_{(t,x)}$ which is in $L^2(\Omega;\cH)$, then $\delta(u) \in \bD^{1,2}$ and the following identity holds, 
\begin{equation}
\label{Heisenberg}
D(\delta(u))=u+\delta(Du).
\end{equation}
\end{proposition}

\section{Estimates on Malliavin Derivatives}
\label{section-Mal}

In this section, we provide some key estimates for the Malliavin derivatives of the solution $u$, which will be used in the proofs of the main results.

We denote by $(u_n)_{n\geq 0}$ the sequence of Picard iterations, given by $u_0(t,x)=1$ and
\begin{equation}
\label{def-Picard}
u_{n+1}(t,x)=1+\int_0^t \int_{\bR^d}G_{t-s}(x-y)b(u_n(s,y))dyds+\int_0^t \int_{\bR^d}G_{t-s}(x-y)\s(u_n(s,y))W(ds,dy).
\end{equation}

\begin{theorem}
\label{key-Du}
For any $0\leq r\leq t \leq T$, $x,z \in \bR^d$ and $p\geq 2$,
\[
\|D_{r,z}u(t,x)\|_p \leq C_{T,p} G_{t-r}(x-z),
\]
where $C_{T,p}>0$ is a constant that depends on $(T,p,q,\|b'\|_\infty,\|\sigma'\|_\infty)$.
\end{theorem}

\begin{proof}
Using Heisenberg commutation rule \eqref{Heisenberg} followed by the chain rule \eqref{chain}, we infer that the sequence $(Du_n)_{n\geq 0}$ satisfies the following recurrence relation:
\begin{align}
\nonumber
D_{r,z}u_{n+1}(t,x)&=G_{t-r}(x-z)\s\big(u_n(r,z)\big)+\int_r^t \int_{\bR^d}G_{t-s}(x-y) b'\big(u_n(s,y)\big)D_{r,z}u_n(s,y) dyds+\\
\label{rec-Dun}
&\quad \int_r^t \int_{\bR^d} G_{t-s}(x-y)\s'\big(u_n(s,y)\big) D_{r,z}u_n(s,y) W(ds,dy),
\end{align}
for any $n\geq 0$. Note that $D_{r,z}u_{0}(t,x)=0$ since $u_0(t,x)=1$. 

We use Burkholder-Davis-Gundy (BDG) inequality and Minkowski inequality to estimate the $p$-th moment of the stochastic integral. Using \eqref{def-KTp}, it follows that for any $0<r<t\leq T$, $x,z \in \bR^d$ and $n\geq 0$,
\begin{align}
\nonumber
&\|D_{r,z}u_{n+1}(t,x)\|_p^2  \leq B_{t} \left\{ G_{t-r}^2(x-z)+  \left(\int_r^t \int_{\bR^d} G_{t-s}(x-y)\|D_{r,z}u_n(s,y)\|_pdyds\right)^2 \right.\\	
\nonumber
& \quad +\left.\int_r^t \int_{(\bR^d)^2} G_{t-s}(x-y) G_{t-s}(x-y') \|D_{r,z}u_n(s,y)\|_p \|D_{r,z}u_n(s,y')\|_p  f(y-y') dydy'ds \right\} \\\
\label{recD1}
& \quad \quad \quad \quad \quad \quad  =: B_{t} \{ G_{t-r}^2(x-z)+  I+II \},
\end{align}
where $B_{t}$ is a constant that depends on $(t,p,\|b'\|_{\infty},\|\sigma'\|_{\infty})$ and is non-decreasing in $t$.

We will prove below that for any $0<r<t\leq T$, $x,z \in \bR^d$ and $n\geq 1$,
\begin{equation}
\label{key-un}
\|D_{r,z}u_{n}(t,x)\|_p \leq C_{T,p} G_{t-r}(x-z).
\end{equation}

The conclusion will follow by Lemma \ref{weaklimit}, since $\{Du_n(t,x)\}_{n\geq 1}$ converges to $Du(t,x)$ in the weak topology $L^2(\Omega;\cH)$ (by Lemma 1.2.3 of \cite{nualart06}).

\medskip

{\em Case 1.} $d=1$.
For the first integral on the RHS of \eqref{recD1}, we use H\"older's inequality with respect to the measure $\mu_t(ds,dy)=G_{t-s}(x-y)dsdy$ on $[r,t] \times \bR$ whose total mass is $(t-r)^2/2$. For the second integral of \eqref{recD1}, we use \eqref{LHS-q}. We obtain that for any $0<r<t\leq T$, $x,z \in \bR, n\geq 0$,
\begin{align*}
\|D_{r,z}u_{n+1}(t,x)\|_p^2  & \leq  B_{T} \left\{ G_{t-r}^2(x-z)+ \frac{T^2}{2}\int_r^t \int_{\bR} G_{t-s}^2(x-y) \|D_{r,z}u_n(s,y)\|_p^2 dyds \right.\\
& \quad \quad \quad \left. c_{f,1}\int_r^t \left( \int_{\bR}G_{t-s}^{2q}(x-y) \|D_{r,z}u_n(s,y)\|_p^{2q}dy\right)^{1/q}ds \right\},
\end{align*}
where $q$ is given by \eqref{def-q}. Using the fact that $G_t^{2q}(x)=c_{q} G_t^{2}(x)$ and applying H\"older's inequality with respect to the measure $G_{t-s}^2(x-y)dy$ whose total mass is $(t-s)/2$, it follows that:
\[
\left(\int_{\bR} G_{t-s}^{2q}(x-y) \|D_{r,z}u_n(s,y)\|_p^{2q} dy\right)^{1/q} \leq c_{q}^{q} \frac{T}{2}\int_{\bR} G_{t-s}^2(x-y)\|D_{r,z}u_n(s,y)\|_p^{2} dy.
\]
Hence
\begin{align*}
&\|D_{r,z}u_{n+1}(t,x)\|_p^2  \leq  A_T \left\{ G_{t-r}^2(x-z)+ \int_r^t \int_{\bR} G_{t-s}^2(x-y) \|D_{r,z}u_n(s,y)\|_p^2 dyds \right\},
\end{align*}
where $A_T=B_{T} \big(1+\frac{T^2}{2}+c_{q}^{q} \frac{T}{2}\big)$.
By induction on $n\geq 1$, it follows that 
\[
\|D_{r,z}u_{n}(t,x)\|_p^2 \leq A_T \left\{ G_{t-r}^2(x-z)+ \sum_{k=1}^{n-1} A_T^k I_k(t,x,r,z)\right\},
\]
for any $0<r<t\leq T$ and $x,z \in \bR$, where
\begin{align}
\label{def-Ik}
I_k(t,x,r,z)&=\int_{r<t_1<\ldots<t_k<t} \int_{\bR^k}\prod_{i=1}^{k}G_{t_{i+1}-t_i}^2(x_{i+1}-x_i) G_{t_1-r}^2(x_1-z)d\pmb{x}_k d \pmb{t}_k,
\end{align}
with $\pmb{x}_k=(x_1,\ldots,x_k)$, $\pmb{t}_k=(t_1,\ldots,t_k)$, $t_{k+1}=t$ and $x_{k+1}=x$. Using the fact that $G_t^2(x)=\frac{1}{2}G_t(x)$ and inequality
\[
\int_{\bR}G_{t-s}(x-y)G_{s-r}(y-z)dy \leq (t-r)G_{t-r}(x-z),
\]
it follows that 
\[
I_k(t,x,r,z)\leq \frac{(t-r)^{2k}}{2^k k!}G_{t-r}(x-z).
\]
Hence, for any $0<r<t\leq T$, $x,z \in \bR$ and $n\geq 1$,
\[
\|D_{r,z}u_{n}(t,x)\|_p^2 \leq 2A_T \left(\sum_{k\geq 0} A_T^k \frac{T^{2k}}{2^k k!} \right) G_{t-r}^2(x-z)=:C_{T,p}^2 G_{t-r}^2(x-z).
\]

\medskip

{\em Case 2.} $d=2$. We proceed similarly to Lemma A.2 of \cite{BS26}. We fix $r \in [0,T]$ and $z \in \bR^2$.
By induction on $n \geq 1$,  $D_{r,z}u_n(t,x)=0$ if $|x-z|\geq t-r$. We denote 
\[
f_n(t,x)=\|D_{r,z}u_n(t,x)\|_p \quad \mbox{and} \quad \widetilde{f}_n(t)=\sup_{|x-z|<t-r}\frac{f_n(t,x)}{G_{t-r}(x-z)}.
\]

We suppose first that $\widetilde{f}_n(t)<\infty$ for all $t \in [r,T]$, and we derive a recurrence relation for $\widetilde{f}_{n+1}$. 

We bound the terms on the right-hand side of \eqref{recD1}. We start with $I$, the term due to the drift. We have:
\begin{align*}
\int_r^t \int_{\bR^2} G_{t-s}(x-y) f_n(s,y)dyds & \leq \int_r^t \widetilde{f}_n(s) \int_{\bR^2} G_{t-s}(x-y) G_{s-r}(y-z) dy ds \\
& = \frac{1}{2}\int_r^t  \widetilde{f}_n(s) \int_{|t+r-2s|}^{t-r} G_{\theta}(x-z)d\theta ds,
\end{align*}
where for the second line we use the  identity \eqref{G*G}.

Considering separately the cases $s \in [r,\frac{t+r}{2}]$ and $s \in [\frac{t+r}{2},t]$, and using Fubini's theorem, it follows that
\[
\int_r^t \int_{\bR^2} G_{t-s}(x-y) f_n(s,y)dyds \leq \frac{1}{2} \int_0^{t-r} \left(\int_{\frac{t+r-\theta}{2}}^{\frac{t+r+\theta}{2}} \widetilde{f}_n(s) ds \right) G_{\theta}(x-z)d\theta.
\]
Using Cauchy-Schwarz inequality, we obtain that for any $\theta \in [0,t-r]$,
\begin{align*}
\int_{\frac{t+r-\theta}{2}}^{\frac{t+r+\theta}{2}} \widetilde{f}_n(s) ds  & \leq \theta^{1/2}
\left(\int_{\frac{t+r-\theta}{2}}^{\frac{t+r+\theta}{2}} \widetilde{f}_n^2(s) ds\right)^{1/2} \leq
\theta^{1/2}
\left(\int_{r}^{t} \widetilde{f}_n^2(s) ds\right)^{1/2}.
\end{align*}
Hence,
\[
\int_r^t \int_{\bR^2} G_{t-s}(x-y) f_n(s,y)dyds \leq \frac{1}{2} \left(\int_{r}^{t} \widetilde{f}_n^2(s) ds\right)^{1/2}\int_0^{t-r} \theta^{1/2} G_{\theta}(x-z)d\theta.
\]
By direct calculation,
\begin{equation}
\label{theta-G}
\int_0^t \theta^{1/2} G_{\theta}(x) d\theta \leq 2t^{3/2} G_{t}(x) \quad \mbox{for any $t>0$}.
\end{equation}
Indeed,
\begin{align*}
\int_0^t \theta^{1/2} G_{\theta}(x) d\theta&=\frac{1}{2\pi}\int_{|x|}^t \sqrt{\frac{\theta}{\theta+|x|}} \cdot \frac{1}{\sqrt{\theta-|x|}} d\theta \leq \frac{1}{2\pi} \int_{0}^{t-|x|} \frac{1}{\sqrt{\eta}} d\eta=\frac{1}{\pi}\sqrt{t-|x|} \\
& \leq 2 \sqrt{(t-|x|)(t^2-|x|^2)} G_t(x) \leq 2t^{3/2} G_{t}(x).
\end{align*}
It follows that
\begin{equation}
\label{bound-I}
I \leq   (t-r)^{3} G_{t-r}^2(x-z) \int_{r}^{t} \widetilde{f}_n^2(s) ds.
\end{equation}

We now bound the term $II$, due to the noise. Using the definition of $\widetilde{f}_n(t)$ and inequality \eqref{LHS-q} with $q$ given by \eqref{def-q}, we have:
\begin{align*}
II & \leq \int_r^t \widetilde{f}_n^2(s) \int_{\bR^2} \int_{\bR^2} G_{t-s}(x-y) G_{t-s}(x-y') G_{s-r}(y-z) G_{s-r}(y'-z) f(y-y') dydy'ds \\
& \leq c_{f,2} \int_r^t \widetilde{f}_n^2(s) \big[ G_{t-s}^{2q}* G_{s-r}^{2q}(x-z)\big]^{1/q}ds.  
\end{align*}
We apply H\"older's inequality, with $\frac{1}{p^*}+\frac{1}{p'}=1$:

\begin{align*}
II \le c_{f,2} \left(\int_r^t \widetilde{f}_n^{2p^*}(s)ds \right)^{1/p^*} \left( \int_r^t \big[ G_{t-s}^{2q}* G_{s-r}^{2q}(x-z)\big]^{p'/q}ds \right)^{1/p'}.
\end{align*}
We now apply Lemma 4.3 of \cite{BNZ} with $\delta=p'/q$, which means that we must choose $1<p'<\frac{q}{2q-1}$. Note that this lemma remains valid for any $\delta(1-2q)+1>0$ and $q \in (\frac{1}{2},1)$. We obtain that
\begin{equation}
\label{bound-II}
II \leq C_{f,q}^2  (t-r)^{\frac{1}{p'}(1-\delta(2q-1))} G_{t-r}^{\frac{1}{p'}\delta(2q-1)}(x-z) \left(\int_r^t \widetilde{f}_n^{2p^*}(s)ds \right)^{1/p^*},
\end{equation}
for some constant $C_{f,q}>1$.

We put together the bounds from \eqref{bound-I} and \eqref{bound-II}. First, by H\"older's inequality, 
\[
\int_r^t \widetilde{f}_n^2(s) ds \leq (t-r)^{1/p'} \left(\int_r^t \widetilde{f}_n^{2p^*}(s)ds \right)^{1/p^*}.
\]
Secondly, we use the fact that $G_{t-r}^{\frac{1}{p'}\delta(2q-1)}(x-z) \les (t-r)^{2-\frac{1}{p'}\delta(2q-1)} G_{t-r}^{2}(x-z)$, due to the trivial bound:
\begin{equation}
\label{G-ab}
G_t^a(x) \leq (2\pi t)^{b-a} G_t^b(x) \quad \mbox{for any $t>0$ and $a<b$}.
\end{equation}
Hence,
\[
f_{n+1}^2(t,x) \le C_{f,q}^2 G_{t-r}^2(x-z) \left\{1+ [(t-r)^{3+\frac{1}{p'}}+(t-r)^{2+\frac{1}{p'}(1-2\delta(2q-1))}] \left( \int_r^t \widetilde f_n^{2p^*}(s)ds\right)^{1/p^*}\right\}
\]
We divide by $G_{t-r}^2(x-z)$ and we take the supremum over $x$ with $|x-z|<t-r$. Then, we take power $p^*$. We obtain that:
\[
\widetilde{f}_{n+1}^{2p^*}(t) \le C_{f,q}^{2p^*} \left\{1+ \left[(t-r)^{3+\frac{1}{p'}}+(t-r)^{2+\frac{1}{p'}(1-2\delta(2q-1))} \right]^{p^*}  \int_r^t \widetilde f_n^{2p^*}(s)ds\right\}.
\]

This relation is similar to (A.10) of \cite{BS26}. Proceeding as in Step 2 and 3 of the proof of Lemma A.1 of \cite{BS26}, we conclude that $\widetilde{f}_{n}(t) \leq C_{f,q} C_t$ where $C_t$ is a constant that is non-decreasing in $t$.
\end{proof}

\begin{corollary}
\label{D2u-support}
For any $p\geq 1$ and almost all $(t,x),(r,z),(\theta,w) \in \bR_{+} \times \bR^d$ with $r,\theta \leq t$,
\[
\|D_{(r,z),(\theta,w)}^2u(t,x)\|_p = \|D_{(r,z),(\theta,w)}^2u(t,x)\|_p 1_{\{ |x-z| <t-r, |x-w| < t - \theta \}} . 
\]
\end{corollary}

\begin{proof}
Let $(t,x),(r,z) \in \bR_{+} \times \bR^d$ be such that $|x-z| \geq t-r$.  By Theorem \ref{key-Du},  $D_{r,z}u(t,x)=0$ a.s. Hence,  
\[
\bE\|D \big(D_{r,z}u(t,x)\big)\|_{\cH}^2=0, 
\]
(Here, we used the fact that if $F \in \bD^{1,2}$ and $F=0$ a.s. then $\bE\|DF\|_{\cH}^2=0$, 
because $D:\bD^{1,2} \to L^2(\Omega;\cH)$ is a linear operator.)

We now apply Lemma \ref{zero-S} to $S=D \big(D_{r,z}u(t,x)\big)$. We infer that for almost all $\theta \in [0,t]$
\[
\bP \big(D_{(\theta,w),(r,z)}^2 u(t,x)=0 \quad \mbox{for almost all $w \in \bR^d$} \big)=1.
\]
In particular, $\|D_{(\theta,w),(r,z)}^2 u(t,x)\|_{p}=0$ for almost all $w \in \bR^d$ and $\theta \in [0,t]$.

From this argument, we deduce that for any $(t,x),(r,z) \in \bR_{+} \times \bR^d$,
\[
1_{\{|x-z| \geq t-r\}} \int_0^t \int_{\bR^d} \|D_{(\theta,w),(r,z)}^2 u(t,x)\|_{p}  dwd\theta=0.
\]
Therefore, integrating $dtdx$ and $drdz$, we obtain that:
\[
\int_{(\bR_{+}\times \bR^d)^3}1_{\{|x-z| \geq t-r\}}1_{\{r\leq t\}} \|D_{(\theta,w),(r,z)}^2 u(t,x)\|_{p} dtdxdrdz d\theta dw =0.
\]
Hence, for almost all $(t,x),(r,z),(\theta,w)$ with $r\leq t$, $\|D_{(\theta,w),(r,z)}^2 u(t,x)\|_{p} 1_{\{|x-z| \geq t-r\}}=0$, or equivalently,
\[
\|D_{(\theta,w),(r,z)}^2 u(t,x)\|_{p} =\|D_{(\theta,w),(r,z)}^2 u(t,x)\|_{p}1_{\{|x-z|< t-r\}}.
\]
The same argument shows that for almost all $(t,x),(r,z),(\theta,w)$ with $\theta \leq t$,
\[
\|D_{(\theta,w),(r,z)}^2 u(t,x)\|_{p} =\|D_{(\theta,w),(r,z)}^2 u(t,x)\|_{p}1_{\{|x-w|< t-\theta\}}.
\]
The conclusion follows.
\end{proof}

Recall definition \eqref{def-q} of $q$.

\begin{theorem}
\label{key-D2u-int}
For any $0\leq r,\theta \leq t' \leq T$, $w \in \bR^d$,  and $p\geq 2$,
\[
\int_{\bR^d} \left[ \sup_{r \vee \theta \leq t \leq t'} \int_{\bR^d} \|D_{(r,z),(\theta,w)}^2u(t,x)\|^{2p'}_p dx \right] dz \leq C_{T,p,p'},
\]
where $p'=1$ if $d=1$ and $f \in L^1(\bR)$ or the noise is white, and  $p' \in (q,1)$ is arbitrary if $d=2$ and $f \in L^1(\bR^2) \cap L^{\ell}(\bR^2)$ for some $\ell>1$, or $f(x)=|x|^{-\beta}$ for some $\beta \in (0,d)$.

\end{theorem}
\begin{proof}
We first note that $D_{r,z}u(t,x) = 0$ whenever $t < r$ by \eqref{Mal-zero}. Furthermore, we use the convention $G_t(x) = 0$ whenever $t < 0$. Let $0\leq r,\theta \leq t \leq t' \leq T$. We denote $\rho = r \vee \theta$.

Taking the first and then second Malliavin derivative of the mild solution \eqref{mild soln} using \eqref{chain} and \eqref{Heisenberg}, 
\begin{align*}
D^2_{(r,z),(\theta,w)}u(t,x) & =  G_{t-r}(x-z) \s'\big(u(r,z)\big) D_{\theta,w} u(r,z) \\
& + G_{t-\theta}(x-w) \s'\big(u(\theta,w)\big) D_{r,z} u(\theta,w) \\
& + \int_\rho^t \int_{\bR^d} G_{t-s}(x-y) b''\big(u(s,y)\big) D_{r,z}u(s,y) D_{\theta,w}u(s,y)  dyds \\
& + \int_\rho^t \int_{\bR^d} G_{t-s}(x-y) b'\big(u(s,y)\big) D_{(r,z),(\theta,w)}^2 u(s,y) dyds \\
& + \int_\rho^t \int_{\bR^d} G_{t-s}(x-y) \s''\big(u(s,y)\big) D_{r,z}u(s,y) D_{\theta,w}u(s,y)  W(ds,dy)\\
& + \int_\rho^t \int_{\bR^d} G_{t-s}(x-y) \s'\big(u(s,y)\big) D_{(r,z),(\theta,w)}^2 u(s,y) W(ds,dy).
\end{align*}
Using BDG inequality and Minkowski's inequality,
\begin{align*}
\|D^2_{(r,z),(\theta,w)}u(t,x)\|_p^{2p'} & \les  G^{2p'}_{t-r}(x-z) \|\s'\big(u(r,z)\big) D_{\theta,w} u(r,z)\|_p^{2p'} \\
& + G^{2p'}_{t-\theta}(x-w) \|\s'\big(u(\theta,w)\big) D_{r,z} u(\theta,w)\|_p^{2p'} \\
& + \left(\int_\rho^t \int_{\bR^d} G_{t-s}(x-y) \|b''\big(u(s,y)\big) D_{r,z}u(s,y) D_{\theta,w}u(s,y)\|_p  dyds\right)^{2p'} \\
& + \left(\int_\rho^t \int_{\bR^d} G_{t-s}(x-y) \|b'\big(u(s,y)\big) D_{(r,z),(\theta,w)}^2 u(s,y)\|_p dyds\right)^{2p'} \\
& + \bigg(\int_\rho^t \int_{\bR^d} \int_{\bR^d} G_{t-s}(x-y) \|\s''\big(u(s,y)\big) D_{r,z}u(s,y) D_{\theta,w}u(s,y)\|_p f(y-y')  \\
& \qquad G_{t-s}(x-y') \|\s''\big(u(s,y')\big) D_{r,z}u(s,y') D_{\theta,w}u(s,y')\|_p dy dy' ds \bigg)^{p'}\\
& + \bigg( \int_\rho^t \int_{\bR^d} \int_{\bR^d} G_{t-s}(x-y) \|\s'\big(u(s,y)\big) D_{(r,z),(\theta,w)}^2 u(s,y)\|_pf(y-y')  \\
& \qquad G_{t-s}(x-y') \|\s'\big(u(s,y')\big) D_{(r,z),(\theta,w)}^2 u(s,y')\|_p dy dy' ds \bigg)^{p'}.
\end{align*}
Using Theorem \ref{key-Du}, H\"older's inequality, and the boundedness of $b'$,$b''$,$\s'$ and $\s''$, 
\begin{equation}
\label{int-D2}
 \int_{\bR^d} \sup_{\rho \leq t \leq t'} \int_{\bR^d} \|D^2_{(r,z),(\theta,w)}u(t,x)\|_{p}^{2p'} dx dz \les \sum_{i=1}^5 \cB_i
\end{equation}
where
\begin{align*}
 \cB_1 & = \int_{\bR^d} \sup_{\rho \leq t \leq t'} \int_{\bR^d} \Big(G^{2p'}_{t-r}(x-z)G^{2p'}_{r-\theta}(z-w) + G^{2p'}_{t-\theta}(x-w)G^{2p'}_{\theta-r}(w-z) \Big) dx dz, \\
 \cB_2 & = \int_{\bR^d} \sup_{\rho \leq t \leq t'} \int_{\bR^d} \left( \int_\rho^t \int_{\bR^d} G_{t-s}(x-y)G_{s-r}(y-z)G_{s-\theta}(y-w) dy ds\right)^{2p'} dx dz, \\
 \cB_3 & = \int_{\bR^d} \sup_{\rho \leq t \leq t'}  \int_{\bR^d} \left( \int_\rho^t \int_{\bR^d} G_{t-s}(x-y) \|D^2_{(r,z),(\theta,w)}u(s,y)\|_{p}  dy ds\right)^{2p'} dx dz, \\
 \cB_4 & = \int_{\bR^d} \sup_{\rho \leq t \leq t'}  \int_{\bR^d} \bigg( \int_\rho^t \int_{\bR^d} \int_{\bR^d} G_{t-s}(x-y)G_{s-r}(y-z)G_{s-\theta}(y-w) f(y-y') \\
 & \quad G_{t-s}(x-y')G_{s-r}(y'-z)G_{s-\theta}(y'-w) dydy' ds\bigg)^{p'} dx dz, \\
 \cB_5 & = \int_{\bR^d} \sup_{\rho \leq t \leq t'}  \int_{\bR^d} \bigg( \int_\rho^t \int_{\bR^d} \int_{\bR^d} G_{t-s}(x-y) \|D^2_{(r,z),(\theta,w)}u(s,y)\|_{p} f(y-y') \\
 & \quad G_{t-s}(x-y') \|D^2_{(r,z),(\theta,w)}u(s,y')\|_{p} dy dy' ds\bigg)^{p'} dx dz. 
\end{align*}

We compute each term individually. For $\cB_1$ we use \eqref{int G p}, noting that $2p'<2$ if $d=2$:
\begin{equation}
\label{B1}
\cB_1 \les \int_{\bR^d} \Big( G^{2p'}_{r-\theta}(z-w) + G^{2p'}_{\theta-r}(w-z)\Big) dz  \les C_T.
\end{equation}

For $\cB_2$ we use H\"older's inequality with respect to the finite measure $\mathbf 1_{\{\rho < s < t\}} \mathbf 1_{\{|x-y| \leq T\}}dyds$ and exponent $2p' > 1$:
\[
 \cB_2 \les \int_{\bR^d} \sup_{\rho \leq t \leq t'}  \int_\rho^t \int_{\bR^d} \int_{\bR^d}  G^{2p'}_{t-s}(x-y)G^{2p'}_{s-r}(y-z)G^{2p'}_{s-\theta}(y-w) dx dy ds dz. \\
\]
Applying \eqref{int G p} with respect to $dx$, 
\[
 \cB_2 \les \int_{\bR^d} \sup_{\rho \leq t \leq t'}  \int_\rho^t \int_{\bR^d}G^{2p'}_{s-r}(y-z)G^{2p'}_{s-\theta}(y-w) dy ds dz \leq \int_\rho^{t'} \int_{\bR^d} \int_{\bR^d} G^{2p'}_{s-r}(y-z)G^{2p'}_{s-\theta}(y-w) dz dy ds.
\]
 Applying \eqref{int G p} with respect to $dz$ and $dy$ then integrating in $ds$, 
\begin{equation}
\label{B2}
 \cB_2 \les \int_\rho^T \int_{\bR^d} \int_{\bR^d} G^{2p'}_{s-\theta}(y-w) dy ds \les \int_\rho^T ds \les C_T.
\end{equation}

Similarly for $\cB_3$, we use H\"older's inequality with respect to the measure $\mathbf 1_{\{\rho \leq s \leq t\}} \mathbf 1_{\{|x-y| \leq T\}}dyds$:
\[
 \cB_3 \les \int_{\bR^d} \sup_{\rho \leq t \leq t'} \int_\rho^t \int_{\bR^d} \int_{\bR^d} G^{2p'}_{t-s}(x-y) \|D^2_{(r,z),(\theta,w)}u(s,y)\|^{2p'}_{p}  dx dy ds dz. \\
\]
Applying \eqref{int G p} with respect to $dx$,
\[
 \cB_3 \les \int_{\bR^d} \sup_{\rho \leq t\leq t'} \int_\rho^t \int_{\bR^d} \|D^2_{(r,z),(\theta,w)}u(s,y)\|^{2p'}_{p} dy ds dz \les \int_{\bR^d} \int_\rho^{t'} \int_{\bR^d} \|D^2_{(r,z),(\theta,w)}u(s,y)\|^{2p'}_{p} dy ds dz. \\
\]
Relabelling, 
\[
 \cB_3 \les \int_\rho^{t'} \int_{\bR^d} \int_{\bR^d} \|D^2_{(r,z),(\theta,w)}u(t,x)\|^{2p'}_{p} dx dz dt. \\
\]
Replacing the $dx$ integral (which depends on $t$) by its supremum over $s \in [\rho, t]$,
\[
 \cB_3 \les \int_\rho^{t'} \int_{\bR^d} \sup_{\rho \leq s \leq t} \int_{\bR^d} \|D^2_{(r,z),(\theta,w)}u(s,x)\|^{2p'}_{p} dx dz dt. 
\]
For $0 \leq t \leq t' \leq T$ we have $1 \leq T^{1-p'}(t'-t)^{p'-1}$. We apply this trivial bound so that $\cB_3$ matches the estimate for $\cB_5$ below, 
\begin{equation}
\label{B3}
 \cB_3 \les \int_\rho^{t'} (t'-t)^{p'-1} \int_{\bR^d} \sup_{\rho<s<t} \int_{\bR^d} \|D^2_{(r,z),(\theta,w)}u(s,x)\|^{2p'}_{p} dx dz dt.
\end{equation}

For $\cB_4$ we use Lemma \ref{Hol-lem} with respect to the measure $\mathbf 1_{\{|x-z| < T\}}dx$, if $p'< 1$. If $p'  = 1$, the relation below is trivial. We obtain: 
\begin{align*}
 \cB_4 & \les \int_{\bR^d} \sup_{\rho \leq t \leq t'}  \bigg(\int_\rho^t \int_{\bR^d} \int_{\bR^d} \int_{\bR^d} G_{t-s}(x-y)G_{s-r}(y-z)G_{s-\theta}(y-w) f(y-y'), \\
 & \quad G_{t-s}(x-y') G_{s-r}(y'-z)G_{s-\theta}(y'-w) dx dy dy' ds \bigg)^{p'} dz.
 \end{align*}
Identifying the convolution $K_{t-s}(y-y')$, taking the supremum over $t \in [\rho,t']$, and using the fact that $t \to K_t(x)$ is non-decreasing, we have:
\begin{align*}
 \cB_4 & \les \int_{\bR^d} \sup_{\rho \leq t \leq t'}  \bigg(\int_\rho^t \int_{\bR^d} \int_{\bR^d} G_{s-r}(y-z)G_{s-\theta}(y-w) f(y-y') K_{t-s}(y-y') \\
 & \quad G_{s-r}(y'-z)G_{s-\theta}(y'-w)  dy dy' ds \bigg)^{p'} dz \\
 & \leq \int_{\bR^d} \bigg(\int_\rho^{t'} \int_{\bR^d} \int_{\bR^d} G_{s-r}(y-z)G_{s-\theta}(y-w) f(y-y') K_{t'}(y-y') \\
 & \quad G_{s-r}(y'-z)G_{s-\theta}(y'-w)  dy dy' ds \bigg)^{p'} dz.
\end{align*}
We apply again Lemma \ref{Hol-lem} with respect to the measure $\mathbf 1_{\{|z-w| < 2T\}}dz$, if $p' < 1$. (If $p'=1$, the relation below is trivial.) We obtain:
\begin{align*}
 \cB_4 & \les \bigg(\int_\rho^{t'} \int_{\bR^d} \int_{\bR^d} \int_{\bR^d} G_{s-r}(y-z)G_{s-\theta}(y-w) f(y-y')  K_{t'}(y-y') \\
 & \quad G_{s-r}(y'-z)G_{s-\theta}(y'-w) dz dy dy' ds \bigg)^{p'}.
\end{align*}
Identifying the convolution $K_{s-r}(y-y')$, and bounding it by $K_{t'}(y-y')$, we obtain:
\begin{align*}
 \cB_4 & \les \bigg(\int_\rho^{t'} \int_{\bR^d} \int_{\bR^d} G_{s-\theta}(y-w) G_{s-\theta}(y'-w) f(y-y') K_{t'}^2(y-y')  dy dy' ds \bigg)^{p'}.
\end{align*}
By Lemma \ref{fK int},
\begin{align*}
 \cB_4 & \les \bigg(\int_\rho^{t'} \bigg( \int_{\bR^d} G^{2p'}_{s-\theta}(y-w) dy \bigg)^{\frac{1}{p'}} ds\bigg)^{p'}.
\end{align*}
Using \eqref{int G p} then integrating with respect to $ds$,
\begin{equation}
\label{B4}
 \cB_4 \les \bigg(\int_\rho^{T} (s-\theta)^{\frac{2-2p'}{p'}} ds\bigg)^{p'} \les C_{T}.
\end{equation}

For $\cB_5$, we note that Corollary \ref{D2u-support} introduces the indicator $\mathbf 1_{\{|x-z| < T\}}$. We apply again Lemma \ref{Hol-lem} for the measure $\mathbf 1_{\{|x-z| < T\}}dx$, if $p'< 1$. (If $p'=1$, the relation below is trivial.) We obtain:
\begin{align*}
  \cB_5 & \les \int_{\bR^d} \sup_{\rho \leq t \leq t'}  \bigg(\int_\rho^t  \int_{\bR^d} \int_{\bR^d} \int_{\bR^d} G_{t-s}(x-y) \|D^2_{(r,z),(\theta,w)}u(s,y)\|_{p} f(y-y') \\
 & \quad G_{t-s}(x-y') \|D^2_{(r,z),(\theta,w)}u(s,y')\|_{p} dx dy dy' ds \bigg)^{p'} dz \\
 & \leq \int_{\bR^d} \bigg(\int_\rho^{t'}  \int_{\bR^d} \int_{\bR^d}  \|D^2_{(r,z),(\theta,w)}u(s,y)\|_{p} f(y-y')K_{t'}(y-y')  \\
 & \quad \|D^2_{(r,z),(\theta,w)}u(s,y')\|_{p} dy dy' ds \bigg)^{p'} dz.
\end{align*}
Applying Lemma \ref{fK int}, we get: 
\begin{align*}
  \cB_5 & \les \int_{\bR^d} \bigg(\int_\rho^{t'}  \bigg(\int_{\bR^d}  \|D^2_{(r,z),(\theta,w)}u(s,y)\|^{2p'}_{p} dy \bigg)^{\frac{1}{p'}} ds \bigg)^{p'} dz.
\end{align*}
The $dy$ integral depends on $s$. Taking the supremum of similar integrals for $s' \in [\rho,s]$, we obtain:
\begin{align*}
  \cB_5 & \les \int_{\bR^d} \bigg( \int_\rho^{t'}  \bigg(\sup_{\rho \leq s'\leq s} \int_{\bR^d}  \|D^2_{(r,z),(\theta,w)}u(s',y)\|^{2p'}_{p} dy \bigg)^{\frac{1}{p'}} ds \bigg)^{p'} dz. 
\end{align*}

We define $$M_{r,z,\theta,w}(t) = \sup_{\rho \leq s \leq t} \int_{\bR^d}  \|D^2_{(r,z),(\theta,w)}u(s,y)\|^{2p'}_{p} dy.$$ 

Then $M_{r,z,\theta,w}(t)$ is non-decreasing in $t$. Therefore, by Lemma \ref{incr ineq}, we have: 
\begin{align}
\nonumber
  \cB_5 & \les  \int_{\bR^d} p'\int_\rho^{t'} (t'-s)^{p'-1} M_{r,z,\theta,w}(s) ds dz \\
  \label{B5}
  & \les \int_\rho^{t'} (t'-t)^{p'-1} \int_{\bR^d} \sup_{\rho \leq s \leq t} \int_{\bR^d}  \|D^2_{(r,z),(\theta,w)}u(s,y)\|^{2p'}_{p} dy dz dt.
\end{align}

We return to \eqref{int-D2}. Combining the bounds \eqref{B1}, \eqref{B2}, \eqref{B3}, \eqref{B4} and \eqref{B5} gives the recursion, 
\begin{align*}
 & \int_{\bR^d} \sup_{\rho \leq t \leq t'} \int_{\bR^d} \|D^2_{(r,z),(\theta,w)}u(t,x)\|_{p}^{2p'} dx dz \\
 & \quad \leq A_T + B_T\int_\rho^{t'} (t'-t)^{p'-1} \int_{\bR^d} \sup_{\rho \leq s \leq t} \int_{\bR^d}  \|D^2_{(r,z),(\theta,w)}u(s,y)\|^{2p'}_{p} dy dz dt
\end{align*}
Let 
\[
 H_{r,\theta,w}(t) = \int_{\bR^d} \sup_{\rho \leq s \leq t} \int_{\bR^d}  \|D^2_{(r,z),(\theta,w)}u(s,y)\|^{2p'}_{p} dy dz.  
\]
Applying the above recursion twice and integrating yields, 
\[
H_{r,\theta,w}(t')  \leq A_T + B_T \int_\rho^{t'} (t'-t)^{2p'-1} H_{r,\theta,w}(t)dt  \leq   A_T + B_T \int_\rho^{t'} H_{r,\theta,w}(t)dt. 
\]
Applying the classical Gronwall lemma gives the desired bound, 
\[
 H_{r,\theta,w}(t) \leq A_Te^{B_T(t-\rho)} \leq C_T.
\]
\end{proof}

The main purpose of Theorem \ref{key-D2u-int} is the following corollary.  

\begin{corollary}
\label{cor2}
For any $0\leq r,\theta \leq t \leq T$, $w \in \bR^d$ and $p\geq 2$,
\[
\int_{\bR^d} \int_{\bR^d} \|D_{(\theta,w),(r,z)}^2 u(t,x)\|_p dxdz \leq C_{T,p}.
\]
\end{corollary}

\begin{proof}
By Corollary \ref{D2u-support}, the integrand contains the indicator of the set 
\[
S_w=\{(x,z) \in (\bR^d)^2;|x-z|<t-r, |z-w|<2t-r-\theta\}.
\] We apply H\"older's inequality with exponent $2p'$ (with $p'$ as in Theorem \ref{key-D2u-int}), with respect to the finite measure $\mu(dx,dz)= 1_{S_w}(x,z) dxdz$. The conclusion follows by Theorem \ref{key-D2u-int}.

\end{proof}

\begin{remark}
{\rm In the case $d = 1$ it is possible to establish the bound 
\[
 \|D^2_{(r,z),(\theta,w)} u(t,x)\|_p \les G_{t-r}(x-z)G_{2t-r-\theta}(z-w)
\]
 for all $0<\theta<r<t<T$ and $x,z,w \in \bR$, which is stronger than Theorem \ref{key-D2u-int}. However, this relies on the wave kernel $G_t$ being in $L^2(\bR^d)$ which only holds in dimension $d = 1$. The method is similar to the proof of Theorem \ref{key-Du}. We omit the proof for brevity. }
\end{remark}

\section{Proof of the main results}
\label{section-main}

In this section, we present the proofs of Theorems \ref{ergodic-th}, \ref{cov-th}, \ref{QCLT} and \ref{FCLT}.

\subsection{Ergodicity}

The proof of Theorem \ref{ergodic-th} follows directly from Theorem \ref{key-Du} and the more general result given below.

\begin{lemma}\label{ergodicity}
Let $\{Z(t,x) : t \geq 0,\ x \in \bR^d \}$ be an adapted random field such that $\{ Z(t,x)\}_{x \in \bR^d}$ is strictly stationary, and $Z(t,x) \in \bD^{1,2}$ for any $t \geq 0$ and $x \in \bR^d$. Assume that for any $0 < r < t$ and $x,z \in \bR^d$,
\begin{equation}\label{DZ}
\| D_{r,z}Z(t,x)\|_2 \leq C_t G_{t-r}(x-z)
\end{equation}
where $C_t > 0$ is a constant that depends on $t$. Then $\{ Z(t,x)\}_{x \in \bR^d}$ is ergodic.
\end{lemma}
\begin{proof} By Lemma 4.2 of \cite{BZ24}, it suffices to show that for any $k\geq 1$, and $b_1, \ldots , b_k, \zeta_1, \ldots, \zeta_k\in\bR^d$ and $g \in \{x \to \sin(x), x \to \cos(x)\}$, the following relation holds:
\begin{align}
\label{ERG1}
\lim_{R\to\infty} \frac{1}{R^{2d}}
{\rm Var}\bigg(  \int_{[0,R]^d} g \bigg(\sum_{j=1}^k b_j  Z(t,x + \zeta_j)  \bigg) dx \bigg)
=0.
\end{align}
We denote 
\[
X_{b,\zeta}(t,x) := g \left(\sum_{j=1}^k b_j  Z(t,x + \zeta_j)  \right)
\]
and 
\[
 V_R := {\rm Var}\left(  \int_{[0,R]^d} X_{b,\zeta}(t,x) dx \right)
\]
By Poincar\'e inequality \eqref{Poincare}, \eqref{LHS-q'}, and Minkowski's inequality,
\begin{align*}
V_R & \leq \int_{\bR_+} \int_{\bR^{2d}} \left\| D_{r,y}\int_{[0,R]^d} X_{b,\zeta}(t,x) dx  \right\|_2 \left\| D_{r,y'}\int_{[0,R]^d} X_{b,\zeta}(t,x) dx  \right\|_2 f(y-y') dydy'dr \\
 & \les \int_{\bR_+} \left[\int_{\bR^{d}} \left\| \int_{[0,R]^d} D_{r,y}X_{b,\zeta}(t,x) dx  \right\|_2^{2q'} dy\right]^{1/q'} dr \\
 & \les \int_{\bR_+} \left[\int_{\bR^{d}} \left( \int_{[0,R]^d} \left\| D_{r,y}X_{b,\zeta}(t,x)  \right\|_2 dx\right) ^{2q'} dy\right]^{1/q'} dr.
\end{align*}
By the definition of $X_{b,\zeta}(t,x)$, the chain rule \eqref{chain}, and \eqref{DZ},
\[
 \left\| D_{r,y}X_{b,\zeta}(t,x)  \right\|_2 \leq C_t\sum_{j=1}^k |b_j|  G_{t-r}(x + \zeta_j - y).
\]
Therefore, by the above and H\"older's inequality, 
\begin{align*}
V_R & \les  \int_{\bR_+} \left[\int_{\bR^{d}} \left( \int_{[0,R]^d} \sum_{j=1}^k |b_j| G_{t-r}(x + \zeta_j - y) dx\right) ^{2q'} dy\right]^{1/q'} dr \\
 & \les \sum_{j=1}^k |b_j|^2 \int_{\bR_+} \left[\int_{\bR^{d}} \left( \int_{[0,R]^d} G_{t-r}(x + \zeta_j - y) dx\right) ^{2q'} dy\right]^{1/q'} dr.
\end{align*}
For $R \geq t$, since $|x| \leq \sqrt{d}R$, we have $|\zeta_j - y| \leq (1+\sqrt{d})R$ for each $j = 1,\cdots,k$. Therefore,
\[
V_R \les \sum_{j=1}^k |b_j|^2 \int_{\bR_+} \left[\int_{B_{(1+\sqrt{d})R}(\zeta_j)} \left( \int_{[0,R]^d} G_{t-r}(x + \zeta_j - y) dx\right) ^{2q'} dy\right]^{1/q'} dr.
\]
Bounding $[0,R]^d$ by $\bR^d$ and integrating in the order of $dx$, $dy$ then $dr$ yields, 
\[
V_R \les \sum_{j=1}^k |b_j|^2 R^{\gamma}.
\]
Since $\gamma < 2d$ this proves Lemma \ref{ergodicity} and in turn, Theorem \ref{ergodic-th}. 
\end{proof}

\subsection{Limiting covariance}

Let $\omega_d={\rm Vol}(B_1)$, i.e. $\omega_1=2$ and $\omega_2=\pi$.

\medskip

{\em Proof of Theorem \ref{cov-th}:} {\em Case 1. $f \in L^1(\bR^d)$.} By Lemma 18 of \cite{dalang99}, the process $\{u(t,x);t\geq 0,x\in \bR^d\}$ satisfies property (S) of \cite{dalang99}. In particular, this implies that for any $(t_1,x_1),\ldots,(t_k,x_k) \in  \bR_{+} \times \bR^d$ and $z \in \bR^d$,
\[
\big(u(t_1,x_1+z),\ldots,u(t_k,x_k+z)\big) \stackrel{d}{=} \big(u(t_1,x_1),\ldots,u(t_k,x_k)\big),
\]
where $\stackrel{d}{=}$ denotes equality in distribution. In particular, for any $t_1,t_2>0$,
\[
{\rm Cov}\big(u(t_1,x),u(t_2,y)\big)=:\rho_{t_1,t_2}(x-y) \quad \mbox{depends only on $x-y$}.
\]
By the dominated convergence theorem,
\begin{align*}
\frac{1}{R^{d}}\bE[F_R(t_1)F_R(t_2)] & =\frac{1}{R^{d}}\int_{B_R} \int_{B_R} \rho_{t_1,t_2}(x-y)dxdy=\frac{1}{R^{d}} \int_{B_{2R}} \rho_{t_1,t_2}(z) {\rm Vol} \big(B_R \cap B_R(-z)\big)dz \\
& \longrightarrow \omega_{d} \int_{\bR^d} \rho_{t_1,t_2}(z) dz=:K(t_1,t_2), \quad \mbox{as} \ R \to \infty,
\end{align*}
provided that 
\begin{equation}
\label{rho-int}
\int_{\bR^d} | \rho_{t_1,t_2}(x)| dx<\infty.
\end{equation}

To prove \eqref{rho-int}, we use Poincar\'e inequality \eqref{Poincare}, followed by the key estimate given by Theorem \ref{key-Du} and relation \eqref{int-G}:

\begin{align*}
& \int_{\bR^d}| \rho_{t_1,t_2}(x)| dx \leq \int_0^{t_1\wedge t_2} \int_{\bR^d} \int_{\bR^d} \int_{\bR^d}  \|D_{r,z}u(t_1,x)\|_2 \|D_{r,z'}u(t_2,0)\|_2 f(z-z') dzdz'dx dr \\
& \leq C_{t_1,t_2,2}^2 \int_0^{t_1\wedge t_2} \int_{\bR^d} \int_{\bR^d} \int_{\bR^d} G_{t_1-r}(x-z) G_{t_2-r}(-z') f(z-z') dzdz'dx dr\\
& \leq C_{t_1,t_2,2}^2\|f\|_{L^1(\bR^d)} \int_0^{t_1\wedge t_2} (t_1-r)(t_2-r) dr<\infty.
\end{align*}

{\em Case 2. $f(x)=|x|^{-\beta}$ for some $\beta \in (0,d\wedge 2)$.}
By the Clark-Ocone formula \eqref{clark-ocone},
\[
F_R(t)=\int_0^t \int_{\bR^d}  \Phi_{t,R}(r,z) W(dr,dz) \quad \mbox{where} \quad \Phi_{t,R}(s,y)=\bE[D_{r,z}F_R(t) | \cF_r].
\]
Assume that $t_1 \leq t_2$. Then
\begin{align*}
\bE[F_R(t_1)F_R(t_2)]  &  =   \int_0^{t_1} \int_{\bR^d} \int_{\bR^d} \bE\big[ \Phi_{t,R}(r,z) \Phi_{t,R}(r,z') \big] |z-z'|^{-\beta}dzdz'dr.
\end{align*}
Consider the following partition: $\bR^d = E_1 \cup E_2 \cup E_3$, where 
\begin{align*}
 E_1 = \{z;|z| < R- t_2\}, \quad
 E_2  = \{z ; R - t_2 \leq  |z| \leq R + t_2\}, \quad
 E_3  = \{z  ; |z|>R +  t_2  \}. 
\end{align*}
Then, 
\[
\bE[F_R(t_1)F_R(t_2)] = \sum_{i,j = 1}^3 \cT_{i,j} (R),
\]
where
\[
 \cT_{i,j}(R) = \int_0^{t_1} \int_{E_i} \int_{E_j} 
 \bE\big[ \Phi_{t,R}(r,z) \Phi_{t,R}(r,z') \big] 
 |z-z'|^{-\beta} dzdz'dr.
\]
By Jensen's and Minkowski's inequalities, and Theorem \ref{key-Du}, for any $r \in [0,t]$ and $z \in \bR^d$,
\begin{equation}
\label{Phi-bound}
\|\Phi_{t,R}(r,z)\|_2 \leq \int_{B_R}\|D_{r,z}u(t,x)\|_2 dx \leq C_t \int_{B_R} G_{t-r}(x-z)dx=C_t \varphi_{t,R}(r,z).
\end{equation}
This implies that $\Phi_{t,R}(r,z)=0$ if $|z|>t-r+R$. By Cauchy-Schwarz inequality,
\begin{equation}
\label{Tij}
 \left|\cT_{i,j}(R)\right| 
  \leq\int_0^{t_1}  \int_{E_i} \int_{E_j}  \| \Phi_{t_1,R}(r,z) \|_2   \|\Phi_{t_2,R}(r,z') \|_2 |z-z'|^{-\beta} dzdz'dr.
\end{equation}
Since
$\|  \Phi_{t_1,R}(r,z)  \|_2 =0$ for all $z\in E_3$, and $\|  \Phi_{t_2,R}(r,z')  \|_2 =0$ for all $z'\in E_3$, it follows that $\cT_{i,j}(R)=0$ if $i=3$ or $j=3$. Hence,
\[
\bE[F_R(t_1)F_R(t_2)] = \sum_{i,j = 1}^2 \cT_{i,j} (R).
\]

We now consider $\cT_{i,2}$ for $i = 1,2$. We use \eqref{Tij} and \eqref{Phi-bound}. Since $\varphi_{t,R}(r,z) \leq t-r$,
\begin{align*}
 \left|\cT_{i,2}(R)\right| & \les \int_0^{t_1}  \int_{E_i} \int_{E_2} (t_1-r)(t_2-r) |z-z'|^{-\beta}dzdz'dr  \les \int_{B_{R + t_2}} \int_{E_2} |z-z'|^{-\beta}dzdz'.
\end{align*}
By \eqref{LHS-q'}, 
\begin{align*}
& \int_{B_{R + t_2}} \int_{E_2} |z-z'|^{-\beta}dzdz'  \leq \left(\int_{B_{R + t_2}} dz\right)^{\frac{2d-\beta}{2d}} \left(\int_{E_2} dz'\right)^{\frac{2d-\beta}{2d}} \\
 & \quad =\left( \omega_d (R + t_2)^{d}\right)^{\frac{2d-\beta}{2d} } \left( \omega_d (R+ t_2)^d - \omega_d (R- t_2)^d\right)^{\frac{2d-\beta}{2d}} \les R^{(2d-1)\frac{2d-\beta}{2d}},
\end{align*}
Therefore, for $i = 1,2$,  
\[
\frac{1}{R^{2d-\beta}} \cT_{i,2} (R) \les R^{-\frac{2d-\beta}{2d}}\xrightarrow[R \to \infty]{} 0. 
\]
Similarly, $\frac{1}{R^{2d-\beta}} \ \cT_{2,1} (R) \to 0$ as $R \to \infty$. 

It remains to consider $\cT_{1,1}(R)$. We write
\[
 \cT_{1,1}(R) = \cT_{1,1}'(R) + \cT_{1,1}'' (R),
\]
where
\begin{align}
\nonumber
 \cT_{1,1}' (R)& = \int_0^{t_1} \int_{E_1} \int_{E_1} {\rm Cov}\left(\Phi_{t_1,R}(r,z), \Phi_{t_2,R}(r,z')\right) |z-z'|^{-\beta} dz dz'dr, \\
 \label{def-T11-s}
\cT_{1,1}''(R) & = \int_0^{t_1} \int_{E_1} \int_{E_1} \bE\big[\Phi_{t_1,R}(r,z)\big]\bE\big[ \Phi_{t_2,R}(r,z')\big] |z-z'|^{-\beta} dzdz'dr.
\end{align}

We first treat $\cT_{1,1}'(R)$.
As in the proof of Proposition 3.1 in \cite{BNZ}, to show that 
\begin{equation}
\label{T11-prime}
\lim_{R \to \infty}
\frac{1}{R^{2d-\beta}}\cT_{1,1}'(R) = 0,
\end{equation}
it suffices to prove that:
\begin{equation}
\label{cov-Phi}
 \lim_{|z-z'| \to \infty } \sup_{r\in [0,t_1]} \sup_{R > 0}{\rm Cov}\big(\Phi_{t_1,R}(r,z), \Phi_{t_2,R}(r,z')\big) = 0. 
\end{equation}
We omit the details of why \eqref{cov-Phi} implies \eqref{T11-prime}. 

We prove \eqref{cov-Phi}. By Clark-Ocone \eqref{clark-ocone},
\[
 \Phi_{t,R}(r,z) = \bE[\Phi_{t,R}(r,z)] + \int_0^r \int_{\bR^d} \bE[D_{\theta,w}\Phi_{t,R}(r,z)  | \cF_{\theta}] W(d\theta,dw).
\]
Therefore, by It\^o isometry, for any $r \in [0,t_1]$ and $R > 0$, 
\begin{align*}
 & {\rm Cov}\left(\Phi_{t_1,R}(r,z), \Phi_{t_2,R}(r,z')\right) \\
 & \quad = \int_0^{r} \int_{\bR^d} \int_{\bR^d} \bE\Big[\bE[D_{\theta,w}\Phi_{t_1,R}(r,z) | \cF_{\theta}]  \bE[D_{\theta,w'}\Phi_{t_2,R}(r,z')| \cF_{\theta}]\Big] |w-w'|^{-\beta}dwdw'd\theta.
\end{align*}
By Cauchy-Schwarz inequality and Jensen's inequality, 
\begin{align*}
 & \left|{\rm Cov}\left(\Phi_{t_1,R}(r,z), \Phi_{t_2,R}(r,z')\right)\right| \\
 & \quad \leq \int_0^{r} \int_{\bR^d} \int_{\bR^d} \|D_{\theta,w}\Phi_{t_1,R}(r,z)\|_2 \|D_{\theta,w'}\Phi_{t_2,R}(r,z')\|_2 |w-w'|^{-\beta}dwdw'd\theta.
\end{align*}
Similarly to Lemma A.2 in \cite{BJQ17}, it can be proved that $D \bE[F|\cF_t]={\bf 1}_{[0,t]}\bE[DF|\cF_t]$. Hence,
\[
 D_{\theta,w} \Phi_{t,R}(r,z)= D_{\theta,w}\big(\bE[D_{r,z}F_R(t)|\cF_r]\big)={\bf 1}_{\{\theta \leq r\}} \bE[D_{(\theta,w),(r,z)}^2 F_R(t) \ | \cF_r]
\]
and therefore, by Jensen's inequality and Minkowski's inequality,
\[
 \|D_{r,z}\Phi_{t,R}(r,z)\|_2 \leq {\bf 1}_{\{\theta \leq r\}}  \|D^2_{(\theta,w),(r,z)}F_R(t)\|_2 \leq {\bf 1}_{\{\theta \leq r\}} \int_{B_R} \|D^2_{(\theta,w),(r,z)} u(t,x) \|_2 dx.
\]
Bounding the integral on $B_R$ by the integral on $\bR^d$, we obtain: 
\begin{align*}
 & \left|{\rm Cov}\left(\Phi_{t_1,R}(r,z), \Phi_{t_2,R}(r,z')\right)\right| \\
 & \quad \leq \int_0^{r} \int_{(\bR^{d})^4}  \|D^2_{(\theta,w),(r,z)} u(t_1,x_1) \|_2 \|D^2_{(\theta,w'),(r,z')} u(t_2,x_2) \|_2 |w-w'|^{-\beta}dx_1dx_2 dw dw' d\theta.
\end{align*}
By Corollary \ref{D2u-support}, the integrand contains the indicator of $|z-w|<2t_1,|z'-w'|<2t_2$. On this set,
\[
|z-z'| \leq |z-w|+|w-w'|+|w'-z'| \leq 4t_2+|w-w'|
\]
Hence, 
\begin{align*}
 & \left|{\rm Cov}\left(\Phi_{t_1,R}(r,z), \Phi_{t_2,R}(r,z')\right)\right| \\
& \leq  (|z-z'| - 4t_2)^{-\beta} \int_0^{t_1} \int_{(\bR^{d})^4} \|D^2_{(\theta,w),(r,z)} u(t_1,x_1) \|_2  \|D^2_{(\theta,w'),(r,z')} u(t_2,x_2) \|_2 dx_1dx_2dw dw'd\theta. 
\end{align*}
Applying Corollary \ref{cor2} proves \eqref{cov-Phi}.

Finally, we study the asymptotic behaviour of $\cT_{1,1}''(R)$. We first examine $\bE[\Phi_{t_1,R}(r,z)]$. 
Observe that for any $z \in E_1$, $r \in [0,t_1]$ and $t' \in [0,t_2]$, 
\[
\bE\left[\int_{B_R^c} D_{r,z}u(t',x)dx\right]=0.
\]
Indeed, by Cauchy-Schwarz inequality and Theorem \ref{key-Du},
\begin{align*}
& \left| \bE\left[\int_{B_{R}^c} D_{r,z}u(t',x) dx \right] \right|  
 \leq \int_{ B_{R}^c} \|D_{r,z}u(t',x)\|_2 dx  \leq C_{t'} \int_{B_{R}^c} G_{t'-r}(x-z)dx,
\end{align*}
and since the integrand contains the indicator of $|x-z|<t'-r$, which is included in $B_R$, since $|x|<|x-z|+|z|<t'-r+R-t_2 \leq R$.
Therefore, for any $z \in E_1$, $r \in [0,t_1]$ and $t' \in [0,t_2]$. 
\begin{align*}
 \bE[\Phi_{t',R}(r,z)] &= \bE[D_{r,z} F_R(t')]=\bE \left[ \int_{B_R} D_{r,z} u(t',x)dx\right]=\bE\left[\int_{\bR^d} D_{r,z}u(t',x) dx \right] \\
& =\bE\left[\int_{\bR^d} D_{r,0}u(t',x-z) dx \right] = \bE\left[\int_{\bR^d} D_{r,0}u(t',x) dx \right] =: m(t',r),
\end{align*}
where  for the first equality on the second line, we used Lemma \ref{transl}. We will use this relation for $t'=t_1$ and $t'=t_2$. Notice that the right side is independent of $R$ and $z$. Returning to the definition \eqref{def-T11-s} of $\cT_{11}''(R)$, we obtain that:
\begin{align*}
 \cT_{1,1}''(R) & =  \int_0^{t_1} m(t_1,r)m(t_2,r) \int_{E_1^2}  |z-z'|^{-\beta} dzdz'dr  =  C_{d,\beta} (R-t_2)^{\gamma}\int_0^{t_1} m(t_1,r)m(t_2,r) dr,
\end{align*}
where $C_{d,\beta} = \int_{B_1^2}  |z-z'|^{-\beta} dzdz'$. Finally, we define: 
\[
 K(t_1,t_2) := C_{d,\beta}\int_0^{t_1} m(t_1,r)m(t_2,r) dr. 
\]
Note that $K(t_1,t_2) < \infty$, by Cauchy-Schwarz inequality and Theorem \ref{key-Du}. Therefore,
\[
\lim_{R \to \infty}\frac{1}{R^{2d-\beta}}\cT_{1,1}'' (R)=K(t_1,t_2). 
\]
Since all other terms vanish as $R \to \infty$, we obtain the desired conclusion \eqref{def-K}.

\subsection{Quantitative CLT}

In this section, we give the proof of Theorem \ref{QCLT}. We recall that $|\cH|$ is the set of all measurable functions $\varphi: \bR_{+}\times \bR^d \to \bR$ such that
\[
\int_{\bR_{+}} \int_{(\bR^d)^2}|\varphi(t,x) \varphi(t,x')| f(x-x') dxdx'dt<\infty,
\]
and $|\cH^{\otimes 2}|$ is the set of all measurable functions $\varphi: (\bR_{+}\times \bR^d)^2 \to \bR$ such that
\[
\int_{\bR_{+}^2} \int_{(\bR^d)^4}|\varphi(t,x,s,y) \varphi(t,x',s,y')| f(x-x') f(y-y')dxdx'dydy'dt ds<\infty.
\]

We will use the following result, which is a particular case of Proposition 1.8 of \cite{BNQZ} with $\gamma_0=\delta_0$.

\begin{proposition}
\label{prop18}
   If $F\in\mathbb{D}^{2,4}$  has mean zero and variance $\sigma^2\in(0,\infty)$ such that with probability 1,  $DF\in| \cH|$ and $D^2F\in|\cH^{\otimes 2}|$, then
\[
d_{\rm TV}\Big( \frac{F}{\sigma},  Z\Big) \leq \frac{4}{\sigma^2} \sqrt{\mathcal{A}},
\]
where $Z\sim N(0,1)$ and
\begin{align*}
\mathcal{A}:&= \int_{\bR_+^3\times\bR^{6d}} \| D_{(r,z),(\theta,w)}^2 F \|_4  \| D_{(s,y),(\theta,w')}^2 F \|_4  \| D_{r,z'}F\|_4 \| D_{s, y' }F \|_4  \\
&\quad  \quad \quad f(y-y')f(z-z') f(w-w')  dydy'dzdz' dwdw'  dr ds d\theta.
\end{align*}
\end{proposition}

\bigskip

{\em Proof of Theorem \ref{QCLT}:} In the case $R < t$ the total variation is bounded by $1$ so $d_{TV}\left(\frac{F_R(t)}{\sigma_R(t)},Z\right) \leq t^{-(\gamma-2d)/2}R^{(\gamma-2d)/2}$. Hence, we may assume that $R \geq t$. By Proposition \ref{prop18},
\[
d_{\rm TV}\Big( \frac{F_R(t)}{\sigma_R(t)},  Z\Big) \leq \frac{4}{\sigma_R^2(t)} \sqrt{\mathcal{A}_R},
\]
where $Z\sim N(0,1)$ and
\begin{align*}
\mathcal{A}_R:&= \int_{[0,t]^3 \times\bR^{6d}}  \| D_{(r,z),(\theta,w)}^2 F_R(t) \|_4  \| D_{(s,y),(\theta,w')}^2 F_R(t)\|_4  \| D_{r,z'}F_R(t)\|_4 \| D_{s, y' }F_R(t) \|_4  \\
&\quad  \quad \quad f(y-y') f(z-z') f(w-w') dydy'dzdz' dwdw'  dr ds d\theta.
\end{align*}
By Theorem \ref{cov-th}, $\sigma_R^2(t)\sim C R^{\gamma}$. Therefore, it is enough to prove that 
\begin{equation}
\label{bound-AR}
\cA_R \les R^{3\gamma-2d}
\end{equation}
By Minkowski's inequality,
\begin{align*}
\mathcal{A}_R & \leq  \int_{[0,t]^3 \times\bR^{6d}} \int_{B_R^4} \| D_{(r,z),(\theta,w)}^2 u(t,x_1) \|_4  \| D_{(s,y),(\theta,w')}^2 u(t,x_2)\|_4  \| D_{r,z'}u(t,x_3)\|_4 \| D_{s, y' }u(t,x_4) \|_4  \\
&\quad  \quad \quad f(y-y') f(z-z') f(w-w') d\pmb{x}_4 dydy'dzdz' dwdw'  dr ds d\theta \\
\end{align*}
where $\pmb{x}_4=(x_1,x_2,x_3,x_4)$. We apply Theorem \ref{key-Du},
\begin{align*}
\cA_{R} & \les \int_{[0,t]^3 \times\bR^{6d}} \int_{B_R^4} \| D_{(r,z),(\theta,w)}^2 u(t,x_1) \|_4  \| D_{(s,y),(\theta,w')}^2 u(t,x_2)\|_4 \\
& \quad  \quad G_{t-r}(x_3-z') G_{t-s}(x_4-y') f(y-y') f(z-z') f(w-w') d\pmb{x}_4 dydy'dzdz' dwdw'  dr ds d\theta.
\end{align*}
By the definition of $\varphi$ given in \eqref{def-var}, we write $\cA_{R}$ as, 
\begin{align*}
\cA_{R} & \les \int_{[0,t]^3 \times\bR^{2d}} f(w-w')  \\
& \qquad \int_{\bR^d} \int_{B_R}\| D_{(r,z),(\theta,w)}^2 u(t,x_1) \|_4 dx_1 \int_{\bR^d} \varphi_{t,R}(r,z') f(z-z') dz' dz \\
& \qquad \int_{\bR^d} \int_{B_R} \| D_{(s,y),(\theta,w')}^2 u(t,x_2)\|_4  dx_2 \int_{\bR^d} \varphi_{t,R}(s,y') f(y-y') dy' dy \\
& \qquad  dwdw'  dr ds d\theta.
\end{align*}
By Corollary \ref{D2u-support}, $\| D_{(r,z),(\theta,w)}^2 u(t,x_1) \|_4$ has support $|x_1-z| < t$. Therefore, the integrand with respect to $dz$ is supported on $B_{R+t}$. Moreover, for $R \geq t$ we have $B_{R+t} \subseteq B_{2R}$. The same logic for the integral with respect to $dy$ yields, 
\begin{align*}
\cA_{R} & \les \int_{[0,t]^3 \times\bR^{2d}} f(w-w')  \\
& \qquad \int_{B_{2R}} \int_{B_R}\| D_{(r,z),(\theta,w)}^2 u(t,x_1) \|_4 dx_1 \int_{\bR^d} \varphi_{t,R}(r,z') f(z-z') dz' dz \\
& \qquad \int_{B_{2R}} \int_{B_R} \| D_{(s,y),(\theta,w')}^2 u(t,x_2)\|_4  dx_2 \int_{\bR^d} \varphi_{t,R}(s,y') f(y-y') dy' dy \\
& \qquad  dwdw'  dr ds d\theta.
\end{align*}

Since $0 \leq \varphi_{t,R}(r,z') \leq t-r$ we have that, 
\begin{align*}
\sup_{z \in B_{2R}} \int_{B_{2R}}\varphi_{t,R}(r,z')f(z-z')dz' \leq (t-r)\int_{B_{4R}}f(z')dz' \les (t-r)R^{\gamma-d}.
\end{align*}
Applying this twice to $\cA_{R}$ gives,
\begin{align*}
\cA_{R} & \les R^{2\gamma-2d}\int_{[0,t]^3} (t-r)(t-s) \int_{\bR^d}\int_{\bR^d}f(w-w') \left(\int_{B_{2R}} \int_{B_R} \| D_{(r,z),(\theta,w)}^2 u(t,x_1) \|_4 dx_1 dz\right) \\
& \qquad \left(\int_{B_{2R}} \int_{B_R} \| D_{(s,y),(\theta,w')}^2 u(t,x_2)\|_4 dx_2 dy\right) dwdw'drdsd\theta.
\end{align*}
By Corollary \ref{D2u-support}, $\| D_{(r,z),(\theta,w)}^2 u(t,x_1) \|_4$ has support $|x_1-w| < t$. Therefore, the integrand with respect to $dw$ is supported on $B_{R+t}$. Moreover, for $R \geq t$ we have $B_{R+t} \subseteq B_{2R}$. The same logic for the integral with respect to $dw'$ yields, 
\begin{align*}
\cA_{R} & \les R^{2\gamma-2d}\int_{[0,t]^3} (t-r)(t-s) \int_{B_{2R}}\int_{B_{2R}}f(w-w') \left(\int_{B_{2R}} \int_{B_R} \| D_{(r,z),(\theta,w)}^2 u(t,x_1) \|_4 dx_1 dz\right) \\
& \qquad \left(\int_{B_{2R}} \int_{B_R} \| D_{(s,y),(\theta,w')}^2 u(t,x_2)\|_4 dx_2 dy\right) dwdw'drdsd\theta.
\end{align*}
Applying \eqref{LHS-q'} and bounding $B_R, B_{2R}$ by $\bR^d$ gives, 
\begin{align*}
\cA_{R} & \les R^{2\gamma-2d}\int_{[0,t]^3}(t-r)(t-s) \left\| \int_{\bR^d} \int_{\bR^d} \| D_{(r,z),(\theta,w)}^2 u(t,x_1) \|_4 dx_1 dz \right\|_{L^{2q'}(B_{2R})} \\
& \quad \left\| \int_{\bR^d}\int_{\bR^d} \| D_{(s,y),(\theta,w')}^2 u(t,x_2)\|_4 dx_2 dy \right\|_{L^{2q'}(B_{2R})} drdsd\theta.
\end{align*}
By Corollary \ref{cor2}, 
\begin{align*}
\cA_{R} & \les R^{2\gamma-2d}\int_{[0,t]^3}(t-r)(t-s) \left(\int_{B_{2R}} C_T^{2q'} dw\right)^{1/2q'}  \left(\int_{B_{2R}} C_T^{2q'} dw'\right)^{1/2q'}  drdsd\theta. \\
& \les R^{3\gamma-2d} \int_{[0,t]^3} (t-r)(t-s) drdsd\theta \les R^{3\gamma-2d}.
\end{align*}
\qed

\subsection{Functional CLT}

In this section, we give the proof of Theorem \ref{FCLT}.

{\em Step 1. (finite-dimensional convergence)} In this step, we show that for any integer $m\geq 1$ and for any $t_1,\ldots,t_m>0$,
\[
\Big(\frac{1}{R^{\gamma/2}} F_R(t_1),\ldots, \frac{1}{R^{\gamma/2}}F_R(t_m)  \Big)
\stackrel{d}{\to} \Big(\cG(t_1),\ldots, \cG(t_m) \Big) \quad \mbox{as $R \to \infty$}.
\]
By Cram\'er-Wold theorem, this is equivalent to showing that for any $b_1,\ldots,b_m \in \bR$,
\[
X_R:=\frac{1}{R^{\gamma/2}}\sum_{j=1}^{m}b_jF_R(t_j) \stackrel{d}{\to}X:=\sum_{j=1}^{m}b_j \cG(t_j) 
\quad \mbox{as $R \to \infty$}.
\]
As in the proof of Theorem 1.5.(ii) of \cite{BS26}, it is enough to prove that $X_R/\tau_R \stackrel{d}{\to} Z$ as $R \to \infty$, where $\tau_R^2={\rm Var}(X_R) \to \tau^2={\rm Var}(X)$ and $Z \sim N(0,1)$. For this, we apply Proposition \ref{prop18} to $F=X_R$. It follows that
\begin{equation}
\label{bound-BR}
d_{\rm TV}\Big( \frac{X_R}{\tau_R},  Z\Big) \leq \frac{4}{\tau_R^2} \sqrt{\cB_R} \leq C \sqrt{\cB_R},
\end{equation}
where 
\begin{align*}
\cB_R:&= \int_{\bR_+^3\times\bR^{6d}} dr ds d\theta  dzdz' dydy' dwdw'  \\
&\quad  f(z-z') f(w-w') f(y-y')  \| D_{(r,z),(\theta,w)}^2 X_R \|_4  \| D_{(s,y),(\theta,w')}^2 X_R \|_4  \| D_{r,z'} X_R\|_4 \| D_{s, y' }X_R \|_4.
\end{align*}

Applying Minkowski's inequality to each of the four $\|\cdot\|_4$-norms above, it follows that
\[
\cB_R \leq \frac{1}{R^{2\gamma}} \sum_{i,j,k,\ell=1}^{m} \big|b_i b_j b_k b_{\ell}\big| \cA_R(t_i,t_j,t_k,t_{\ell}), 
\]
where
\begin{align*}
\cA_R(t_i,t_j,t_k,t_{\ell}) &= \int_{B_R^4} dx_1 dx_2 dx_3 dx_4 \int_{\bR_+^3 \times \bR^{6d}}  dr ds d\theta  dzdz' dydy' dwdw'  f(z-z') f(w-w') f(y-y')  \\
&\quad   \| D_{(r,z),(\theta,w)}^2 u(t_i,x_1) \|_4  \| D_{(s,y),(\theta,w')}^2 u(t_j,x_2) \|_4  \| D_{r,z'} u(t_k,x_3)\|_4 \| D_{s, y'} u(t_{\ell},x_4) \|_4.
\end{align*}
Using the same method as for the proof of \eqref{bound-AR}, it can be proved that
\[
\cA_R(t_i,t_j,t_k,t_{\ell}) \leq C R^{3\gamma-2d}
\]
Therefore, it follows that:
\[
\cB_R \leq CR^{\gamma-2d}
\]
Combining this with \eqref{bound-BR}, we infer that $d_{\rm TV}\Big( \frac{X_R}{\tau_R},  Z\Big) \to 0$ as $R \to \infty$.

\medskip

{\em Step 2. (tightness and H\"older continuity)}  We  will use Kolmogorov-Chentsov criterion (Theorem 3.23 of \cite{kallenberg02}) for the existence of the H\"older continuous modification, and the classical moment criterion (Corollary 16.9 of \cite{kallenberg02}) for tightness.  

Note that
\begin{align*}
u(t,x)-\bE[u(t,x)]& =\int_0^t \int_{\bR^d} G_{t-r}(x-z) \sigma\big(u(r,z)\big) W(dr,dz)+\\
&\int_0^t \int_{\bR^d} G_{t-r}(x-z) \Big(b\big(u(r,z)\big) -\bE\big[b\big(u(r,z)\big) \big]\Big) drdz.
\end{align*}

Integrating $dx$ on $B_R$, and using the definition of $\varphi_{t,R}$ \eqref{def-var} , we obtain that,
\begin{align*}
F_R(t) & =\int_0^t \int_{\bR^d} \varphi_{t,R}(r,z) \sigma\big(u(r,z)\big) W(dr,dz)+\\
&\int_0^t \int_{\bR^d}  \varphi_{t,R}(r,z)  \Big(b\big(u(r,z)\big) -\bE\big[b\big(u(r,z)\big) \big]\Big) drdz.
\end{align*}

Let $s,t \in [0,T]$ with $s<t$. We write $F_R(t)-F_R(s)=\sum_{i=1}^4 T_i$, where
\begin{align*}
T_1 & = \int_0^s \int_{\bR^d} \Big(\varphi_{t,R}(r,z) - \varphi_{s,R}(r,z) \Big) \sigma\big(u(r,z)\big) W(dr,dz)\\
T_2 &= \int_0^s \int_{\bR^d}  \Big( \varphi_{t,R}(r,z) - \varphi_{s,R}(r,z) \Big) \Big(b\big(u(r,z)\big) -\bE\big[b\big(u(r,z)\big) \big]\Big) drdz\\
T_3&=\int_s^t \int_{\bR^d} \varphi_{t,R}(r,z) \sigma\big(u(r,z)\big) W(dr,dz)\\
T_4&=\int_s^t \int_{\bR^d}  \varphi_{t,R}(r,z)  \Big(b\big(u(r,z)\big) -\bE\big[b\big(u(r,z)\big) \big]\Big) drdz.
\end{align*}
To estimate $T_1$ and $T_3$, we use Burkholder-Davis-Gundy inequality, followed by Minkowski inequality and the uniform moment estimate \eqref{def-KTp}. We obtain that for any $p\geq 2$,
\begin{align*}
\|T_1\|_p^2 & \les \int_0^s \int_{\bR^d} \int_{\bR^d}\Big|\varphi_{t,R}(r,y) - \varphi_{s,R}(r,y) \Big|  \Big|\varphi_{t,R}(r,z) - \varphi_{s,R}(r,z) \Big| f(y-z)dydz dr \\
\|T_3\|_p^2 & \les \int_s^t \int_{\bR^d} \int_{\bR^d} \varphi_{t,R}(r,y) \varphi_{t,R}(r,z) f(y-z) dydzdr.
\end{align*}
 
Applying \eqref{LHS-q'}, \eqref{est} and \eqref{diff est} gives
\begin{align*}
\|T_1\|_p^2 & \les \int_0^s \| \varphi_{t,R}(r,\cdot) - \varphi_{s,R}(r,\cdot)  \|_{L^{2q'}(\bR^d)}^2 dr \les R^{\gamma}(t-s)^{2/d}\\
\|T_3\|_p^2 & \les \int_s^t \| \varphi_{t,R}(r,\cdot) \|_{L^{2q'}(\bR^d)}^2 dr \les R^{\gamma}(t-s)^3 \les  R^{\gamma}(t-s)^{2/d}.
\end{align*}
For $T_2$ and $T_4$, we use Clark-Ocone formula, noting that $T_2$ is $\cF_s$-measurable, $T_4$ is $\cF_t$-measurable, and $\bE[T_2] = \bE[T_4] = 0$,
\begin{align*}
T_2 & = \int_0^s \int_{\bR^d} \bE[D_{r',y} T_2 |\cF_{r'}] W(dr',dy), \\
T_4 & = \int_0^t \int_{\bR^d}\bE[D_{r',y} T_4  | \cF_{r'}] W(dr',dy). 
\end{align*}

Similarly, by Burkholder-Davis-Gundy and Minkowski inequalities, for any $p\geq 2$,
\begin{align}
\label{T2}
\|T_2\|^2_p & \les \int_0^s \int_{\bR^d} \int_{\bR^d} \|D_{r',y} T_2\|_{p} \|D_{r',w} T_2\|_{p} f(y-w) dy dw dr', \\
\label{T4}
\|T_4\|^2_p & \les \int_0^t \int_{\bR^d} \int_{\bR^d} \|D_{r',y} T_4\|_{p} \|D_{r',w} T_4\|_{p} f(y-w) dy dw dr'. 
\end{align}
We now calculate $\|D_{r',y} T_2\|_{p}$ and $\|D_{r',y} T_4\|_{p}$. Using the chain rule \eqref{chain} for the Malliavin derivative, and relation \eqref{Mal-zero}, we have:
\begin{align*}
D_{r',y}T_2 &= \int_{r'}^s \int_{\bR^d}  \Big( \varphi_{t,R}(r,z) - \varphi_{s,R}(r,z) \Big) b'\big(u(r,z)\big) D_{r',y}u(r,z) drdz\\
D_{r',y}T_4&=\int_{s \vee r'}^t   \int_{\bR^d}  \varphi_{t,R}(r,z)  b'\big(u(r,z)\big) D_{r',y}u(r,z) drdz.
\end{align*}
Applying Minkowski's inequality, the fact that $b'$ is bounded, and Theorem \ref{key-Du}, we obtain that:
\begin{align*}
\|D_{r',y}T_2\|_{p} &\leq \int_{r'}^s \int_{\bR^d}  \left| \varphi_{t,R}(r,z) - \varphi_{s,R}(r,z) \right| \ \|D_{r',y}u(r,z)\|_{p}  drdz\\
&\les \int_{r'}^s  \int_{\bR^d}  \left| \varphi_{t,R}(r,z) - \varphi_{s,R}(r,z) \right| G_{r-r'}(z-y)  drdz\\
\|D_{r',y}T_4\|_p &\leq  \int_{s \vee r'}^t  \int_{\bR^d}  \varphi_{t,R}(r,z) \|D_{r',y}u(r,z)\|_{p} drdz \\
 &\les \int_{s \vee r'}^t \int_{\bR^d}  \varphi_{t,R}(r,z) G_{r-r'}(z-y) drdz. 
\end{align*}
Substituting the above bounds into \eqref{T2}, \eqref{T4} and applying \eqref{LHS-q'} gives, 
\begin{align*}
\|T_2\|^2_p & \les \int_0^s \left(\int_{\bR^d} \left( \int_{r'}^s \int_{\bR^d}  \left| \varphi_{t,R}(r,z) - \varphi_{s,R}(r,z) \right| G_{r-r'}(z-y)  drdz \right)^{2q'}  dy \right)^{1/q'}dr' \\
\|T_4\|^2_p & \les \int_0^t \left(\int_{\bR^d} \left( \int_{s \vee r'}^t \int_{\bR^d}  \varphi_{t,R}(r,z) G_{r-r'}(z-y)  drdz \right)^{2q'}  dy \right)^{1/q'}dr'.
\end{align*}
Note that $2q' > 1$ and $1/q' \geq 1$ so we can apply H\"older's inequality twice to the $dr$ integral, bounding the mass by $T$, to obtain, 
\begin{align*}
\|T_2\|^2_p & \les \int_0^s \int_{r'}^s \left(\int_{\bR^d} \left( \int_{\bR^d}  \left| \varphi_{t,R}(r,z) - \varphi_{s,R}(r,z) \right| G_{r-r'}(z-y)  dz \right)^{2q'}  dy \right)^{1/q'}drdr' \\
\|T_4\|^2_p & \les \int_s^t \int_{0}^r \left(\int_{\bR^d} \left(  \int_{\bR^d}   \varphi_{t,R}(r,z) G_{r-r'}(z-y)  dz \right)^{2q'}  dy \right)^{1/q'}dr'dr.
\end{align*}
By Young's inequality, 
\begin{align*}
\|T_2\|^2_p & \les \int_0^s \int_{r'}^s  \| \varphi_{t,R}(r,\cdot) - \varphi_{s,R}(r,\cdot) \|_{L^{2q'}(\bR^d)}^2 \|G_{r-r'}\|_{L^1(\bR^d)}^2 drdr' \\
\|T_4\|^2_p & \les \int_s^t \int_0^r  \| \varphi_{t,R}(r,\cdot)\|_{L^{2q'}(\bR^d)}^2 \|G_{r-r'}\|_{L^1(\bR^d)}^2 dr'dr.
\end{align*}
Finally, we apply \eqref{est}, \eqref{diff est} and $\|G_{r-r'}\|_{L^1(\bR^d)} = r-r'$, to obtain the estimate:
\begin{align*}
\|T_2\|^2_p & \les \int_0^s \int_{r'}^s  R^{\gamma} (t-s)^{\frac{2}{d}} (r-r')^2 drdr' \les R^{\gamma}(t-s)^{2/d} \\
\|T_4\|^2_p & \les \int_s^t \int_0^r R^{\gamma} (t-r)^{2} (r-r')^2 dr'dr \les R^{\gamma}(t-s)^{3} \les R^{\gamma}(t-s)^{2/d}.
\end{align*}
Hence, for any $p\geq 2$, $0\leq s<t\leq T$ and $R>T$,
\[
\|F_R(t)-F_R(s)\|_p \les R^{\gamma/2} (t-s)^{1/d}.
\]


\appendix
\section{Some auxiliary results}

In this section, we include some auxiliary results which are used in the paper.

\medskip

The following result was used in the proof of Corollary \ref{D2u-support}.
\begin{lemma}
\label{zero-S}
If $S \in L^2(\Omega;\cH)$ is such that $\bE\|S\|_{\cH}^2=0$, then
$\bP\big(S(t,\cdot)=0\big)=1$ for almost all $t \geq 0$.
If in addition, $S(t,\cdot)$ is a function, then
\[
\bP\big(S(t,x)=0 \ \mbox{for almost all $x \in \bR^d$} \big)=1 \quad \mbox{for almost all $t \geq 0$}.
\]
\end{lemma}

\begin{proof} Since
\[
\bE\|S\|_{\cH}^2=\bE \int_0^{\infty} \int_{\bR^d} |\cF S(t,\cdot)(\xi)|^2\mu(d\xi) dt=0,
\]
and the Lebesgue measure is absolutely continuous with respect to $\mu$, there exists a set $A \in \cF \otimes \cB(\bR_{+})$ with $(\bP\times {\rm Leb})(A)=0$ such that for all $(\omega,t) \in A^c$, $\cF S(\omega,t,\cdot)(\xi)=0$ for almost all $\xi \in \bR^d$. Therefore, 
$A^c \subseteq \{(\omega,t)\in \Omega \times \bR_{+};S(\omega,t,x)=0 \text{ for almost all }x \in\bR^d\}$.

Let $A_t=\{\omega \in \Omega;(\omega,t) \in A\}$ be the $t$-section of $A$, for any $t\geq 0$. By Fubini theorem,
\[
0=(\bP\times {\rm Leb})(A)=\int_0^{\infty} \bP(A_t)dt,
\]
and so, $\bP(A_t)=0$ for almost all $t\geq 0$. But $A_t^c \subseteq 
\{S(t,x)=0 \text{ for almost all } x \in \bR^d\}$. Hence, $\bP\big(S(t,x)=0 \text{ for almost all } x \in \bR^d\big)=1$ for almost all $t \geq 0$.
\end{proof}

In the proof of Theorem \ref{key-D2u-int}, we use the following application of H\"older's inequality, which allows us to move the exponent $p' \in (0,1)$ {\em outside} the integral.

\begin{lemma}
\label{Hol-lem}
Let $(X,\mathcal{X},\mu)$ be a finite measure space and $\phi:X \to \bR_{+}$ be a measurable function. For any $p' \in (0,1)$,
\[
\int_{X} \big(\phi(x)\big)^{p'} \mu(dx) \leq \big[\mu(X)\big]^{1-p'} \left( \int_{X} \phi(x) \mu(dx)\right)^{p'}.
\]
\end{lemma}

\begin{proof}
Let $p=\frac{1}{p'}>1$ and $F(x)=\big(\phi(x)\big)^{p'}$. By H\"older's inequality,
\[
\left(\int_{X} F(x) \mu(dx)\right)^p \leq \big[\mu(X)\big]^{p-1}  \int_{X} \big(F(x) \big)^p \mu(dx).
\]
The conclusion follows taking power $1/p$.
\end{proof}

In the proof of Theorem \ref{key-D2u-int}, we also use the following, which allows us to move the exponent $p' \in (\frac12,1]$ {\em inside} the integral, at a cost. 

\begin{lemma} \label{incr ineq}
Let $p' \in (\frac12,1]$ and $0 \leq r<t$. Let $M:[r,t] \to [0,\infty)$ be a non-decreasing right-continuous function. Then
\[
 \left( \int_r^t M(s)^{\frac{1}{p'}} ds \right)^{p'} \leq p' \int_r^t (t-s)^{p'-1} M(s) ds.
\]
\end{lemma}

\begin{proof} 
The $p'= 1$ case is trivial.
We suppose that $p'\in (\frac12,1)$.  By a change of variable,
\[
\left( \int_r^t M(s)^{\frac{1}{p'}} ds \right)^{p'} =\left( \int_0^{t-r} M(t-s)^{\frac{1}{p'}} ds \right)^{p'}.
\]
We denote $H(s)=M(t-s)$ for $s \in [0,t-r]$ and $L=t-r$. Then $H$ is non-increasing and left-continuous. It is enough to prove that
\begin{equation}
\label{H-ineq}
J:=\left( \int_0^L H(s)^{\frac{1}{p'}} ds \right)^{p'} \leq p' \int_0^L s^{p'-1} H(s) ds.
\end{equation}

By Minkowski's inequality,
\begin{align*}
J  &= \left[ \int_0^L \left( \int_0^{\infty} 1_{\{r\leq H(s)\}} dr\right)^{\frac{1}{p'}}ds \right]^{p'} = \left\| \int_0^{\infty}1_{\{r \leq H(s)\}} dr\right\|_{L^{\frac{1}{p'}}([0,L];ds)}\\
& \leq \int_0^{\infty} \|1_{\{r \leq H(s)\}} \|_{L^{\frac{1}{p'}}([0,L];ds)} dr=\int_0^{\infty} \|1_{\{s \leq H^{-1}(r)\}} \|_{L^{\frac{1}{p'}}([0,L];ds)} dr=\int_0^{\infty}\big(H^{-1}(r)\big)^{p'}dr,
\end{align*}
where $H^{-1}(y)=\sup\{x \in [0,L];H(x)\geq y\}$. On the other hand, by Fubini theorem,
\begin{align*}
p' \int_0^L s^{p'-1} H(s) ds&=p' \int_0^L s^{p'-1} \left(\int_0^{\infty} 1_{\{r \leq H(s)\}} dr \right) ds\\
& = p' \int_0^{\infty} \int_0^L s^{p'-1} 1_{\{s \leq H^{-1}(r)\}} dsdr=p' \int_0^{\infty} \left(\int_0^{H^{-1}(r)} s^{p'-1}ds \right) dr\\
&=\int_0^{\infty} \big(H^{-1}(r)\big)^{p'}dr.
\end{align*}
Relation \eqref{H-ineq} follows.
\end{proof}

\begin{lemma}\label{weaklimit}
Let $p\geq 2$ and let $h:\bR_+\times\bR^d\to[0,\infty)$ be a measurable function such that
\[
\int_0^\infty\int_{\bR^d}h(r,z)dzdr<\infty
\]
and
\[
\int_0^\infty\int_{\bR^{2d}}
h(r,z)h(r,z')f(z-z')dzdz'dr<\infty.
\]
Suppose that $\{X_n\}_{n\geq1}$ and $X$ are in $L^2(\Omega;|\cH|)$, and that $X_n$ converges to $X$ in the weak topology of $L^2(\Omega;\cH)$. If
\[
\sup_{n\geq1}\|X_n(r,z)\|_p\leq h(r,z)
\]
for almost all $(r,z)\in\bR_+\times\bR^d$, then
\[
\|X(r,z)\|_p\leq h(r,z)
\]
for almost all $(r,z)\in\bR_+\times\bR^d$.
\end{lemma}

\begin{proof}
Define the measure $\nu$ on $\bR_+\times\bR^d$ by
\[
\nu(dr,dz)=h(r,z)drdz.
\]
By assumption, $\nu$ is a finite measure. For every $n\geq1$, let
\[
Y_n(r,z)=
\begin{cases}
\dfrac{X_n(r,z)}{h(r,z)}, & h(r,z)>0,\\
0, & h(r,z)=0.
\end{cases}
\]
Hence, $X_n=hY_n$ almost everywhere, and
\[
\begin{aligned}
\bE\int_{\bR_+\times\bR^d} |Y_n(r,z)|^p\nu(dr,dz) = \int_{\{h>0\}} \frac{\bE|X_n(r,z)|^p}{h(r,z)^p} \nu(dr,dz) \leq\nu(\bR_+\times\bR^d)<\infty.
\end{aligned}
\]
Since $\{Y_n\}_{n\geq1}$ is bounded in $L^p(\Omega\times\nu)$, there exist a subsequence $\{Y_{n_k}\}_{k\geq1}$ and $Y\in L^p(\Omega\times\nu)$ such that
\[
Y_{n_k}\rightharpoonup Y \quad \text{weakly in } L^p(\Omega\times\nu).
\]

We first show that
\[
\|Y(r,z)\|_p\leq1
\quad\text{for $\nu$-almost all $(r,z)$.}
\]
Indeed, by Proposition 3.5 of \cite{B11}, for any measurable set $A\subseteq\bR_+\times\bR^d$, 
\[
\int_A \bE|Y(r,z)|^p \nu(dr,dz) \leq \liminf_{k\to\infty} \int_A \bE|Y_{n_k}(r,z)|^p \nu(dr,dz) \leq\nu(A).
\]
Since this holds for every measurable set $A$, we conclude that
\[
\|Y(r,z)\|_p\leq1
\quad\text{for $\nu$-almost all $(r,z)$.}
\]

Next, we show that $hY\in L^2(\Omega;|\cH|)$. Since $f$ is nonnegative, by Cauchy--Schwarz inequality,
\begin{align*}
\bE\|hY\|_{|\cH|}^2 & = \int_0^\infty\int_{\bR^{2d}} h(r,z)h(r,z')f(z-z') \bE|Y(r,z)Y(r,z')| dzdz'dr \\
& \leq \int_0^\infty\int_{\bR^{2d}} h(r,z)h(r,z')f(z-z') \|Y(r,z)\|_2\|Y(r,z')\|_2 dzdz'dr \\
& \leq \int_0^\infty\int_{\bR^{2d}} h(r,z)h(r,z')f(z-z') dzdz'dr < \infty.
\end{align*}
Therefore, $hY\in L^2(\Omega;|\cH|)$.

We now show that $X=hY$ in $L^2(\Omega;\cH)$. Let $\varphi\in \cD (\bR_+\times\bR^d)$ and $Z$ be a bounded random variable. By the weak convergence of $X_n$,
\[
\bE[\langle X,Z\varphi\rangle_{\cH}] = \lim_{k\to\infty} \bE[Z\langle X_{n_k},\varphi\rangle_{\cH}].
\]
On the other hand, using the definition of the inner product and the fact that $X_{n_k} = hY_{n_k}$,
\begin{align*}
& \bE[Z\langle X_{n_k},\varphi\rangle_{\cH}] \\
& \quad = \bE\int_{\bR_+\times\bR^d} ZX_{n_k}(r,z)(f*\varphi(r,\cdot))(z) dr dz \\
& \quad = \bE\int_{\bR_+\times\bR^d} ZY_{n_k}(r,z)(f*\varphi(r,\cdot))(z) \nu(dr dz).
\end{align*}
Since $f*\varphi$ is bounded and $\nu$ is finite, we have $Z(f*\varphi)\in L^{p/(p-1)}(\Omega\times\nu)$. Hence, by the weak convergence of $Y_{n_k}$,
\[
\lim_{k\to\infty} \bE[Z\langle X_{n_k},\varphi\rangle_{\cH}] = \bE\int_{\bR_+\times\bR^d} ZY(r,z)(f*\varphi(r,\cdot))(z) \nu(dr,dz) = \bE[Z\langle hY,\varphi\rangle_{\cH}].
\]
Consequently,
\[
\bE[\langle X,Z\varphi\rangle_{\cH}] = \bE[\langle hY,Z\varphi\rangle_{\cH}].
\]
Since the class of such $Z \varphi$ is dense in $L^2(\Omega;\cH)$, we conclude that
\[
X = hY \quad \text{in } L^2(\Omega;\cH).
\]

Finally, since $X,hY\in L^2(\Omega;|\cH|)$, the random field $S=X-hY$ belongs to $L^2(\Omega;|\cH|)$ and satisfies
\[
\bE\|S\|_{\cH}^2 = 0.
\]
By Lemma~\ref{zero-S}, we obtain
\[
X(\omega,r,z) = h(r,z)Y(\omega,r,z)
\]
for almost all
$(\omega,r,z) \in \Omega\times\bR_+\times\bR^d$.

Since $\nu$ and Lebesgue measure restricted to $\{h>0\}$ are equivalent, we have $\|Y(r,z)\|_p\leq1$ for almost all $(r,z)$ such that $h(r,z)>0$. Therefore,
\[
\|X(r,z)\|_p = h(r,z)\|Y(r,z)\|_p \leq h(r,z)
\]
for almost all $(r,z)$ with $h(r,z)>0$. On the other hand, the almost-everywhere equality $X=hY$ implies that $X(r,z)=0$ almost surely for almost all $(r,z)$ such that $h(r,z)=0$.

Consequently,
\[
\|X(r,z)\|_p\leq h(r,z)
\]
for almost all $(r,z)\in\bR_+\times\bR^d$.
\end{proof}

\section{Translation invariance}

In this section, we present a translation invariance property of the Malliavin derivative of $u$, which was used in the proof of Theorem \ref{cov-th} (limiting covariance) for the case $f(x)=|x|^{-\beta}$.
For this result, we will use the following homogeneity property of the stochastic integral: for any random $g$, integrable with respect to $W$, and for any $z \in \bR^d$,
\begin{equation}
\label{gZ}
\int_0^{\infty}\int_{\bR^d}g(t,x)W(dt,dx)=\int_0^{\infty}\int_{\bR^d}g^{(z)}(t,x)W^{(z)}(dt,dx) \quad \mbox{a.s.}
\end{equation}
where $g^{(z)}(t,x)=g(t,x+z)$ and $W_t^{(z)}(B)=W_t(B+z)$ for any $t>0$ and $B \in \cB_b(\bR^d)$.

By Lemma 18 of \cite{dalang99}, it is known that the solution $u$ of \eqref{SWE} {\em has property (S)}, i.e. 
\begin{equation}
\label{lem18}
(u,W) \stackrel{d}{=}(u^{(z)},W^{(z)}) \quad \mbox{for all $z \in \bR^d$},
\end{equation}
where $u=\{u(t,x);t\geq 0,x\in \bR^d\}$, $W=\{W_t(B);t\geq 0,B \in \cB_b(\bR^d)\}$, and $\stackrel{d}{=}$ denotes the equality of finite dimensional distributions.

\begin{lemma}
\label{transl}
For any $z \in \bR^d$,
\[
\{D_{r,z}u(t,x); t\geq 0, x\in \bR^d\} \stackrel{d}{=} \{D_{r,0}u(t,x-z); t\geq 0, x\in \bR^d\}. 
\]
\end{lemma}

\begin{proof}
Letting $n\to \infty$ in \eqref{rec-Dun}, we obtain that
\begin{align*}
D_{r,z}u(t,x)&=G_{t-r}(x-z)\s\big(u(r,z)\big)+\int_r^t \int_{\bR^d}G_{t-s}(x-y) b'\big(u(s,y)\big)D_{r,z}u(s,y) dyds+\\
&\quad \int_r^t \int_{\bR^d} G_{t-s}(x-y)\s'\big(u(s,y)\big) D_{r,z}u(s,y) W(ds,dy).
\end{align*}
In particular, replacing $(x,z)$ by $(x-z,0)$, we obtain:
\begin{align*}
& D_{r,0}u(t,x-z)=G_{t-r}(x-z)\s\big(u(r,0)\big)+\int_r^t \int_{\bR^d}G_{t-s}(x-z-y) b'\big(u(s,y)\big)D_{r,0}u(s,y) dyds+\\
&\quad \quad \quad \int_r^t \int_{\bR^d} G_{t-s}(x-y-z)\s'\big(u(s,y)\big) D_{r,0}u(s,y) W(ds,dy) \\
&=G_{t-r}(x-z)\s\big(u(r,0)\big)+\int_r^t \int_{\bR^d}G_{t-s}(x-y) b'\big(u(s,y-z)\big)D_{r,0}u(s,y-z) dyds+\\
&\quad \quad \quad \int_r^t \int_{\bR^d} G_{t-s}(x-y)\s'\big(u(s,y-z)\big) D_{r,0}u(s,y-z) W^{(-z)}(ds,dy),
\end{align*}
where for the last line we used \eqref{gZ} for the stochastic integral and a change of variable for the $dy$ integral. This means that $D_{r,z}u=\{D_{r,z}u(t,x);t\geq r,x\in \bR^d\}$ is the solution of
\begin{align*}
X(t,x)&=G_{t-r}(x-z)\s\big(Y(r,z)\big)+\int_r^t \int_{\bR^d}G_{t-s}(x-y) b'\big(Y(s,y)\big)X(s,y) dyds+\\
&\quad \int_r^t \int_{\bR^d} G_{t-s}(x-y)\s'\big(Y(s,y)\big) X(s,y) M(ds,dy)
\end{align*}
with input $(Y,M)=(u,W)$, i.e. for $D_{r,z}u=\Phi(u,W)$ for a certain function $\Phi$. Similarly,  $D_{r,0}u^{(-z)}=\Phi (u^{(-z)},W^{(-z)} )$.  By \eqref{lem18}, $(u,W) \stackrel{d}{=}(u^{(-z)},W^{(-z)}) $. The conclusion follows.

\end{proof}

\begin{remark}
{\rm
The Malliavin derivative $D_{r,z}u(t,x)$ has the product representation:
\begin{equation}
\label{prod}
D_{r,z}u(t,x)=\sigma\big(u(r,z)\big) v^{(r,z)}(t,x),
\end{equation}
where $v^{(r,z)}$ is the solution of
\begin{align*}
v^{(r,z)}(t,x)&=G_{t-r}(x-z)+\int_r^t \int_{\bR^d}G_{t-s}(x-y) b'\big(u(s,y)\big)v^{(r,z)} (s,y) dyds+\\
&\quad \int_r^t \int_{\bR^d} G_{t-s}(x-y)\s'\big(u(s,y)\big) v^{(r,z)}(s,y) W(ds,dy) \quad \text{for } t \geq r.
\end{align*}
This follows by uniqueness, since the left side and the right side of \eqref{prod} satisfy the same equation. Furthermore, if $\sigma'$ and $b'$ are constant, then $\sigma\big(u(r,z)\big)$ is independent of $v^{(r,z)}(t,x)$.

Similarly to Lemma \ref{transl}, it can be proved that:
\begin{equation}
\label{transl-v}
\{v^{(r,z)}(t,x);t\geq r,x \in \bR^d\} \stackrel{d}{=} \{v^{(r,0)}(t,x-z);t\geq r,x \in \bR^d\} \quad \mbox{for all $z \in \bR^d$}.
\end{equation}
See also Lemma 2.5 of \cite{BHWXY25}, Proposition 2.3.(i) of \cite{alberts} and Theorem 2.19.(iii) of \cite{COTX26}, for the case $\sigma(u) = u$ and $b(u) = 0$.
(References \cite{alberts,COTX26} deal with the heat equation.)
}
\end{remark}



\end{document}